\documentclass[preprint,3p]{elsarticle}
\usepackage{amssymb}
\usepackage{amsmath}
\usepackage{subfig}
\usepackage[dvipsnames]{xcolor}

\newcommand*{\qedw}{\hfill\ensuremath{\square}}%

\usepackage{blindtext}
\usepackage{hyperref}
\usepackage{mathtools}
\usepackage{lineno}

\newdefinition{theorem}{Theorem}
\newdefinition{proposition}{Proposition}
\newdefinition{lemma}{Lemma}
\newdefinition{corollary}{Corollary}
\newdefinition{definition}{Definition}
\newdefinition{assumption}{Assumption}
\newdefinition{remark}{Remark}
\newdefinition{example}{Example}
\newproof{proof}{Proof}

\newproof{proofLemmaADTAAT}{Proof of Lemma \ref{lemma:ADT/AAT}}
\newproof{proofTheoremPTCL}{Proof of Theorem \ref{thm:PTCL}}

\hypersetup{
    colorlinks=true,
    linkcolor=blue,
    filecolor=magenta,      
    urlcolor=cyan,
    citecolor=orange
    }

\makeatletter
\def\ps@pprintTitle{%
  \let\@oddhead\@empty
  \let\@evenhead\@empty
  \def\@oddfoot{\reset@font\hfil\thepage\hfil}
  \let\@evenfoot\@oddfoot
}
\makeatother

\journal{~}

\begin{document}

\begin{frontmatter}



\title{Prescribed-Time and Hyperexponential Concurrent Learning with Partially Corrupted Datasets: A Hybrid Dynamical Systems Approach\tnoteref{funding}}
\tnotetext[funding]{This work was supported in part by NSF grant ECCS CAREER 2305756 and AFOSR grant FA9550-22-1-0211.}

\author[1]{Daniel E. Ochoa}
\ead{dochoatamayo@ucsc.edu}
\affiliation[1]{organization={University of California Santa Cruz},
addressline={606 Eng. Loop},
postcode={95064},
city={Santa Cruz, CA},
country={USA}}

\author[2]{\texorpdfstring{Jorge I. Poveda\corref{corresponding}}{Jorge I. Poveda}}
\ead{poveda@ucsd.edu}
\affiliation[2]{organization={University of California San Diego},
addressline={9500 Gilman Dr},
postcode={92093},
city={La Jolla, CA},
country={USA}}

\cortext[corresponding]{Corresponding author}

\begin{abstract}
    We introduce a class of concurrent learning (CL) algorithms designed to solve parameter estimation problems with convergence rates ranging from hyperexponential to prescribed-time while utilizing alternating datasets during the learning process. The proposed algorithm employs a broad class of dynamic gains, from exponentially growing to finite-time blow-up gains, enabling either enhanced convergence rates or user-prescribed convergence time independent of the dataset's richness. The CL algorithm can handle applications involving switching between multiple datasets that may have varying degrees of richness and potential corruption. The main result establishes convergence rates faster than any exponential while guaranteeing uniform global ultimate boundedness in the presence of disturbances, with an ultimate bound that shrinks to zero as the magnitude of measurement disturbances and corrupted data decreases. The stability analysis leverages tools from hybrid dynamical systems theory, along with a dilation/contraction argument on the hybrid time domains of the solutions. The algorithm and main results are illustrated via a numerical example.
\end{abstract}


\end{frontmatter}

\section{Introduction}

Concurrent Learning (CL) is a data-driven framework that is suitable for the design of estimation and learning dynamics in a variety of applications where persistence of excitation (PE) conditions are not feasible \cite{chowdhary2010concurrent}. Examples include parameter estimation problems in batteries \cite{ochoa2021accelerated}, exoskeleton robotic systems \cite{casas2023switched}, extremum seeking \cite{poveda2021data}, excavating robots \cite{greene2021simultaneous}, and reinforcement learning  \cite{ochoa2022acceleratedADP,chowdhary2010concurrent}. In these applications, datasets containing past \emph{recorded} measurements of the relevant signals within the systems are typically available for estimation purposes. When these datasets are considered to be ``sufficiently rich'', they can be integrated into dynamic estimation algorithms to achieve (uniform) exponential convergence to the unknown parameters in the absence of PE. While these techniques have been recently enhanced via non-smooth tools to achieve finite-time and fixed-time convergence results \cite{ochoa2021accelerated,rios2017time,tatari2021fixed}, most CL techniques suffer from two main limitations that are common in practice, see \cite[Sec. 4]{ochoa2021accelerated},\cite{casas2023switched}.  First, the rate of convergence is directly related to the ``level of richness'' of the dataset used by the algorithm. This is true in exponentially convergent CL algorithms \cite{chowdhary2010concurrent}, as well as in finite-time and fixed-time convergence approaches \cite{ochoa2021accelerated,rios2017time,tatari2021fixed}. Second, in many practical applications, CL algorithms do not use a single dataset during the learning process, but rather multiple datasets obtained at different time instants and exhibiting different levels of informativity. Moreover, in real-world scenarios, these datasets might be contaminated with measurement noise, sensor failures, or malicious attacks. This challenge is particularly significant in Internet-of-Things (IoT) and edge computing systems, where datasets are transmitted remotely due to limited onboard storage or the need for remote pre-processing. Such remote transmission might expose the data to man-in-the-middle attacks. Incorporating both the switching nature and potential corruption of these datasets into the convergence analysis of CL algorithms is thus an essential task to inject confidence in the applicability of CL techniques in practical settings.\smallbreak

In this paper, we introduce a novel CL architecture that aims to simultaneously tackle the above two challenges. The proposed algorithm has two main features: First, instead of relying on momentum \cite{ochoa2021accelerated}, non-Lipschitz vector fields \cite{ochoa2021accelerated,tatari2021fixed}, or resetting techniques \cite{le2022concurrent} for improved transient performance, we extend our previous work on blow-up gains \cite{switchedPT} to consider a broader class of dynamic gains. These range from exponentially growing gains yielding hyperexponential convergence rates, to a bigger class of finite-time blow-up gains that enable convergence independent of the initial estimate and dataset richness, allowing the user to \emph{either enhance the exponential convergence rate or to fully prescribe the convergence time}. Second, unlike standard results in the literature of continuous-time CL, we incorporate switching between multiple datasets into the algorithm, where datasets exhibit varying degrees of richness and potential corruption. This is achieved using a suitable \emph{data-querying hybrid automaton} that generates switching signals satisfying generalized average dwell-time and activation-time constraints. Building upon our previous work on PT stability in hybrid systems \cite{switchedPT}, these constraints generalize those in the preliminary version of this work \cite{ochoa2024hybridADHS} to accommodate both finite-time blow-up and exponentially growing gains.\smallbreak

The proposed technique is inspired by recent advances in prescribed-time (PT)  \cite{Orlov2022,Song2023} and hyperexponential (HE) \cite{ochoa2024hybridADHS} adaptive control and regulation, which have remained mostly unexplored in the setting of CL. In the proposed Switched Prescribed-Time and Hyperexponential Concurrent Learning (SPTHE-CL) 
algorithms, only \emph{some} of the datasets are assumed to be sufficiently rich, and at any given time only \emph{one} dataset can be used by the algorithm. This setting is common in reinforcement learning in the context of mini-batch optimization \cite{stapor2022mini}, but it has remained unexplored in the context of CL.\smallbreak

Our main result, presented in Theorem \ref{thm:PTCL}, is derived using tools from the hybrid dynamical system's literature \cite{bookHDS}, and a dilation/contraction argument on the hybrid time domains of the solutions, recently introduced in \cite{ochoa2024dynamic} for the analysis of hybrid dynamic inclusions that incorporate dynamic gains. To the best of our knowledge, the proposed scheme is the first CL algorithm that achieves PT and HE convergence using switching datasets with potentially corrupted data. \smallbreak

The rest of this paper is organized as follows. In Section 2 we present the notation and preliminaries on hybrid dynamical systems. Section 3 presents the main SPTHE-CL algorithm and the main result. Section 4 presents the analysis and the proofs. Section 5 presents a numerical example, and Section 6 ends with the conclusions. 

\emph{Related Work:} The results presented in this paper introduce three key modifications over previous contributions by the authors while providing comprehensive proofs that subsume and generalize the preliminary results in the conference version of this work \cite{ochoa2024hybridADHS}. First, we broaden the framework for CL algorithms with dynamic gains presented in the conference version by introducing a more general class of gains that unifies hyperexponential and prescribed-time convergence. Second, we prove convergence guarantees for scenarios with corrupted datasets, addressing a critical gap in existing CL literature. Third, we extend the results from \cite{ochoa2024dynamic} to handle input-to-state stability (ISS) bounds, establishing new connections between target and nominal hybrid dynamics employing dynamic gains.
While our analysis builds upon and extends the theoretical tools for PT-stability in set-valued hybrid dynamical systems developed in \cite{switchedPT}, we emphasize that \cite{switchedPT} neither addresses concurrent learning nor parameter estimation problems. Furthermore, it does not present the unified treatment of hyperexponential and prescribed-time convergence introduced here.

\section{Preliminaries}
Given a closed set $\mathcal{A}\subset\mathbb{R}^n$ and a vector $x\in\mathbb{R}^n$, we use $|x|_{\mathcal{A}}\coloneqq\min_{s\in\mathcal{A}}\|x-s\|_2$ to denote the distance from $x$ to $\mathcal{A}$, where $\|\cdot\|_2$ denotes the 2-norm.  
To simplify notation, for two (or more) vectors $u,v \in \mathbb{R}^{n}$, we write $(u,v)=[u^{\top},v^{\top}]^{\top}$ to denote their concatenation. Also, given a set $\mathcal{O}\subset\mathbb{R}^n$, we use $\mathbb{I}_{\mathcal{O}}(\cdot)$ to denote the indicator function, which satisfies $\mathbb{I}_{\mathcal{O}}(x)=1$ if $x\in \mathcal{O}$, and $\mathbb{I}_{\mathcal{O}}(x)=0$ if $x\notin \mathcal{O}$. A function $\gamma:\mathbb{R}_{\ge 0} \to\mathbb{R}_{\ge0}$ is of class $\mathcal{K}$ if it is continuous, strictly increasing, and satisfies $\gamma(0)=0$. A function $\beta:\mathbb{R}_{\geq0}\times\mathbb{R}_{\geq0}\to\mathbb{R}_{\geq0}$ is of class $\mathcal{K}\mathcal{L}$ if it is nondecreasing in its first argument, nonincreasing in its second argument, $\lim_{r\to0^+}\beta(r,s)=0$ for each $s\in\mathbb{R}_{\geq0}$, and  $\lim_{s\to\infty}\beta(r,s)=0$ for each $r\in\mathbb{R}_{\geq0}$. A function $\tilde{\beta}:\mathbb{R}_{\ge 0}\times \mathbb{R}_{\ge 0}\times \mathbb{R}_{\ge0}\to\mathbb{R}_{\ge 0}$ belongs to class $\mathcal{KLL}$ if for every $s\in\mathbb{R}_{\geq0}$, $\tilde{\beta}(\cdot, s, \cdot)$ and $\tilde{\beta}(\cdot, \cdot, s)$ belong to class $\mathcal{KL}$ \cite{Cai2009}. A set-valued map $F:\mathbb{R}^n\rightrightarrows\mathbb{R}^n$ associates every point in $\mathbb{R}^n$ with a subset of $\mathbb{R}^n$. \medbreak

\emph{Hybrid Dynamical Systems: } To model our algorithms and analyze the influence of dynamic gains while accounting for external disturbances, in this paper, we will work with hybrid dynamical systems (HDS) aligned with \cite{bookHDS,sanfelice2021hybridBook}, and described by the following inclusions:
\begin{subequations}\label{eq:HDS0}
    \begin{align}
        &(z, u)\in C\coloneqq C_z\times \mathbb{R}^m~~~~~~~~~~~\dot{z}\in F(z, u),\label{HDS0:flow}\\[3pt]
        &(z, u)\in D\coloneqq D_z\times \mathbb{R}^m~~~~~~~~~~z^+\in G(z, u), \label{HDS0:jump}
    \end{align}
\end{subequations}
where $z\in\mathbb{R}^n$ is the state of the system, $u\in\mathbb{R}^m$ represents an input or disturbance, $F:\mathbb{R}^{n}\times \mathbb{R}^m\rightrightarrows\mathbb{R}^{n}$ is called the flow map, $G:\mathbb{R}^{n}\times \mathbb{R}^m\rightrightarrows\mathbb{R}^n$ is called the jump map, $C\subset\mathbb{R}^{n}\times \mathbb{R}^{m}$ is called the flow set, and $D\subset\mathbb{R}^{n}\times \mathbb{R}^m$ is called the jump set. To represent a HDS with the above data we use the notation $\mathcal{H}=(C,F,D,G)$.\smallbreak
System \eqref{eq:HDS0} has solutions of the form $(z,u)$, where $z$ is a hybrid arc and $u$ is a hybrid input \cite[Definition 2.27]{sanfelice2021hybridBook}. These solutions are parameterized by two indices: a continuous-time index $t\in\mathbb{R}{\geq0}$, which increases continuously during flows, and a discrete-time index $j\in\mathbb{Z}{\geq0}$, which increments by one at each jump. Therefore, solutions  to \eqref{eq:HDS0} are defined on \emph{hybrid time domains} (HTDs). A subset $E\subset\mathbb{R}_{\geq0}\times\mathbb{Z}_{\geq0}$ is called a \textsl{compact} HTD if $E=\cup_{j=0}^{J-1}([t_j,t_{j+1}],j)$ for some finite sequence of times $0=t_0\leq t_1\ldots\leq t_{J}$. The set $E$ is a HTD if for all $(T,J)\in E$, $E\cap([0,T]\times\{0,\ldots,J\})$ is a compact HTD. For further definitions of solutions to HDS, their connections to HTDs, and further details on models of the form \eqref{eq:HDS0} we refer the reader to \cite[Chapter 2]{bookHDS} and \cite[Chapter 2]{sanfelice2021hybridBook}.\smallbreak
A hybrid solution pair $(z,u)$ satisfies $\text{dom}(z)=\text{dom}(u)$ and is said to be maximal if it cannot be further extended. This does not necessarily imply that $\text{sup}_t\text{dom}(z)=\infty$, or that $\text{sup}_j\text{dom}(z)=\infty$, although at least one of these two conditions should hold when $z$ is complete. We denote by $\mathcal{S}_\mathcal{H}(K)$ the set of all maximal solution pairs $(z,u)$ to $\mathcal{H}$ with $z(0,0)\in K\subset \mathbb{R}^n$. For $K=\{z_0\}$, we write $\mathcal{S}_{\mathcal{H}}(z_0)$, and let $(z,u)\in \mathcal{S}_{\mathcal{H}}$ indicate that $(z,u)$ is a maximal solution pair to $\mathcal{H}$ when no set $K$ is specified.  To simplify notation, in this paper we use $|u|_{(t,j)}=\sup_{\substack{(0,0)\leq(\tilde{t},\tilde{j})\leq (t,j)\\(t,j)\in\text{dom}(z)}}\left|u(\tilde{t},\tilde{j})\right|$, and we use $|u|_{\infty}$ to denote  $|u|_{(t,j)}$ when $t+j\to\infty$. \medbreak 

\emph{Dynamic Gains:} The convergence properties of the proposed class of CL algorithms in this paper will be achieved through the use of a suitable class of dynamic gains. To model the effect of these gains in our algorithms, we use HDS of the form \eqref{eq:HDS0} with state $z=(x,\tau, \mu)$, where $x\in \mathbb{R}^{n_x}$ contains the states of the learning dynamics, including logical and auxiliary states, $\tau\in \mathbb{R}_{\ge 0}$ is used to model the evolution of regular time, and $\mu\in \mathbb{R}_{\ge 1}$ denotes the dynamic gain. Additionally, we specialize the data of the HDS as follows:
\begin{subequations}\label{eq:HDS^Time^Gain}
    \begin{align}
        C_z&\coloneqq C_x\times \mathbb{R}_{\ge 0}\times \mathbb{R}_{\ge 1},\qquad F(z,u)\coloneqq 
            \mu\cdot F_x(x,u)\times \{1\}\times \{F_{\mu,\ell}(\mu)\},\\[3pt]
        D_z&\coloneqq D_x\times \mathbb{R}_{\ge 0}\times \mathbb{R}_{\ge 1},\qquad G(z,u) \coloneqq
            G_x(x,u)\times
            \{\tau\}\times
            \{\mu\},
    \end{align}
\end{subequations}
where the sets $C_x,D_x\subset \mathbb{R}^{n_x}$ and the set-valued maps $F_x,G_x:\mathbb{R}^{n_x}\times \mathbb{R}^m\rightrightarrows \mathbb{R}^{n_x}$ are to be defined, and $\ell\in [1,\infty]$ is a tunable parameter that chooses the evolution of the dynamic gain $\mu$ according to a particular flow-map $F_{\mu,\ell}:\mathbb{R}_{\ge1}\to \mathbb{R}_{\ge 1}$. Given a tunable constant $\Upsilon>0$ and $\ell\in\mathbb{R}_{\ge1}\cup\{\infty\}$, we define this map, for each $\mu\in \mathbb{R}_{\ge1}$, as
\begin{equation}\label{eq:muFlowmaps}
    \begin{split}
       F_{\mu,1}(\mu)&\coloneqq \frac{\mu}{\Upsilon},\\
       F_{\mu,\ell}(\mu)&\coloneqq\frac{\ell}{\ell-1}\frac{\mu^{2-\frac{1}{\ell}}}{\Upsilon},~ \text{for }\ell>1,\text{ and }\\
       F_{\mu,\infty}(\mu)&\coloneqq \lim_{\ell\to \infty}F_{\mu,\ell}(\mu)= \frac{\mu^2}{\Upsilon}.
    \end{split}
\end{equation}
These flow maps generate gains that exhibit behavior ranging from exponential growth ($\ell=1$) to finite-time blow-up ($\ell>1$), with $\ell=\infty$ recovering the standard prescribed-time case studied in the conference version of this paper \cite{ochoa2024hybridADHS}. The following lemma, which follows by direct integration, characterizes the solutions of the differential equations generated by these flow maps.
\begin{lemma}(Dynamic Gains)\label{lemma:DynamicGain}
    Consider the family of gain ordinary differential equations (ODEs): 
    \begin{equation}\label{muODE}
    \mu\in\mathbb{R}_{\geq1},~~~~\dot{\mu} = F_{\mu,\ell}(\mu),
    \end{equation}
    where $F_{\mu,\ell}$ is given in \eqref{eq:muFlowmaps}. For each $\mu_0\in\mathbb{R}_{\geq1}$, the unique solutions to \eqref{muODE} from $\mu_0$ are given by:
    \begin{equation}\label{eq:def:dynamicGains}        
        \begin{split}
            \mu_1(t) &= \mu_0e^{t/\Upsilon},\\
            \mu_\ell(t) &= \frac{\Upsilon^{\tfrac{\ell}{\ell-1}}}{\left(\Upsilon\cdot \mu_0^{\tfrac{1-\ell}{\ell}} - t \right)^{\tfrac{\ell}{\ell-1}}},~~\ell>1,\\
            \mu_\infty(t) &= \frac{\Upsilon}{\Upsilon\cdot\mu_0^{-1} - t},
        \end{split}
    \end{equation}
    for all  $t\in\left[0,T_{\mu_0,\ell}\right)$, where  $T_{\mu_0,0}\coloneqq\infty$, $T_{\mu_0,\ell}\coloneqq\Upsilon\mu_0^{(1-\ell)/\ell}$ for $\ell>1$, $T_{\mu_0,\infty}\coloneqq\Upsilon/\mu_0$, and $t\mapsto\mu_\ell(t)$ is continuous in its domain for all $\ell\in\mathbb{R}_{\ge1}\cup\{\infty\}$. Moreover, $\mu_\ell(t)$ is strictly increasing and satisfies $\lim_{t\to T_{\mu_0,\ell}}\mu_\ell(t)=\infty$ for all $\ell\in \mathbb{R}_{\ge0}\cup\{\infty\}$.\qedw
\end{lemma}
\begin{remark}
    By Lemma \ref{lemma:DynamicGain}, the state $\mu$ evolving according to $F_{\mu,\ell}$ for $\ell>1$ will always exhibit finite escape times at the time $T_{\mu_0,\ell}$. In our algorithms, this is actually a desirable feature since, by Lemma \ref{lemma:DynamicGain}, these finite escape times are ``controlled'' by $\Upsilon$ and $\mu(0)=\mu_0$. In practical implementations, the algorithm is usually stopped before $T_{\mu_0,\ell}$, inducing a small residual error. Therefore, borrowing similar terminology used in the PT-literature \cite{song2017time,Orlov2022,Song2023}, when $\ell>1$, we will refer to $\mu$ as a \emph{blow-up} gain, and to $T_{\mu_0,\ell}$ as the \emph{prescribed time} (PT). In the literature on PT regulation for ODEs \cite{song2017time,Orlov2022}, $\mu_0$ is usually taken to be equal to one. However, for the sake of generality, in this paper, all our results are expressed in terms of $\mu_0\in \mathbb{R}_{\ge 1}$ so that we can characterize the effect of the initial conditions of $\mu$ on the performance of the CL dynamics.
\end{remark}
The following diffeomorphisms $\mathcal{D}_{c,\ell}$ will characterize the time-scale transformations induced by the dynamic gains $\mu$ when employed in our algorithms. Given $c\in \mathbb{R}_{\ge 1}$, $\mathcal{D}_{c,\ell}:\mathbb{R}_{\ge 0}\to [0,T_{c,\ell})$ is defined as:
\begin{equation}\label{eq:def:diffeomorphisms}        
    \begin{split}
    \mathcal{D}_{c,1}(t) &\coloneqq ce^{t/\Upsilon} - c,\\
    \mathcal{D}_{c,\ell}(t) &\coloneqq (\ell-1)\cdot \Upsilon\left(\Upsilon^{\frac{1}{\ell-1}}\cdot(\Upsilon\cdot c^{\frac{1-\ell }{\ell}} - t)^{\frac{1}{1-\ell}} - c^{\frac{1}{\ell}}\right),~~~ \ell>1,\\
    \mathcal{D}_{c,\infty}(t) &\coloneqq \Upsilon\ln\left(\frac{\Upsilon}{\Upsilon -ct}\right),
    \end{split}     
\end{equation}
for all $t\in [0,T_{\mu_0,\ell})$, where $T_{\mu_0,\ell}$ is as defined in Lemma~\ref{lemma:DynamicGain}. For the diffeomorphism $\mathcal{D}_{c,\ell}$, when $\ell=1$, explicit formulas for its inverse and its derivative can be obtained by direct computation, as well as properties such as monotonicity with respect to $t$. The case $\ell>1$ admits the same properties, with their characterization given in \cite[Proposition 1]{switchedPT}. Among these properties, the following relation is fundamental to our results:
\begin{equation}\label{eq:matchingEquation}
   \frac{d}{dt}\mathcal{D}_{\mu_0,\ell}(t) = \left(\hat{\mu}\circ \mathcal{D}_{\mu_0,\ell}\right)(t),\quad \forall t\in [0,T_{\mu_0,\ell}),~\mu_0\in\mathbb{R}_{\ge 1},
\end{equation}
where, for all $\ell\in[1,\infty]$, $\hat{\mu}_{\ell}$ is the unique solution to the ODE $\dot{\hat{\mu}} = \frac{1}{\hat{\mu}}F_{\mu,\ell}(\mu)$ with initial condition $\mu_0$. Moreover, this solution $\hat{\mu}_{\ell}$ is complete for any $\mu_0\in \mathbb{R}_{\ge1}$.

\medbreak
\emph{Stability Notions:} The stability concepts introduced in Definition \ref{def:stability} below expand upon classical input-to-state stability (ISS) concepts by integrating recent advances in prescribed-time and hyperexponential convergence. The prescribed-time approach draws from the work of \cite{switchedPT}, which originally developed this concept for hybrid systems involving blow-up gains. The hyperexponential stability notion is adapted from \cite[Def. 5]{zimenko2023stability} and translated to the hybrid systems domain following similar ideas to those presented in \cite{switchedPT}.\medbreak

\begin{definition}\label{def:stability}
    The set $\mathcal{A}\coloneqq\mathcal{A}_x\times \mathbb{R}_{\ge 0}\times \mathbb{R}_{\ge 1}$ where $\mathcal{A}_x\subset \mathbb{R}^{n_x}$ is compact is said to be:
    \begin{itemize}
        \item \emph{Hyperexponentially Input-to-State Stable via Flows} (HE-ISS$_F$) for the HDS $\mathcal{H}$ in \eqref{eq:HDS^Time^Gain} if there exists a function $\gamma\in\mathcal{K}$, and constants $k_1,k_2>0$ such that for every $z_0=(x_0,\tau_0,\mu_0)\in\left(C_x\cup D_x\right)\times\mathbb{R}_{\ge0}\times \mathbb{R}_{\ge1}$, and every solution-pair $(z,u)\in\mathcal{S}_{\mathcal{H}}(z_0)$ the following bound holds:
        \begin{align}\label{eq:def:stability:HE}
            \displaystyle|z(t,j)|_{\mathcal{A}} &\leq k_1e^{-k_2\mathcal{D}_{\mu_0,1}(t)}e^{-k_2\cdot j}|z_0|_{\mathcal{A}} + \gamma(|u|_{(t,j)}),
        \end{align}
        for all $(t,j)\in\text{dom}(z)$. If this holds with $u\equiv0$, the set $\mathcal{A}$ is said to be \emph{Uniformly Globally Hyperexponentially Stable via Flows} (UGHE-S$_F$).
    
        \item \emph{Prescribed-Time Input-to-State Stable via Flows} (PT-ISS$_F$) for the HDS $\mathcal{H}$ if there exist $\beta\in\mathcal{KLL}$ and $\gamma\in\mathcal{K}$ such that for every $z_0=(x_0,\tau_0,\mu_0)\in \left(C_x\cup D_x\right)\times \mathbb{R}_{\ge1}$, and every solution-pair $(z,u)\in\mathcal{S}_{\mathcal{H}}(z_0)$ the following bound holds:
        \begin{align}\label{eq:def:stability:PT}
            |z(t,j)|_{\mathcal{A}} &\leq \beta(|z_0|_{\mathcal{A}},\mathcal{D}_{\mu_0,\ell}(t),j) + \gamma(|u|_{(t,j)}),
        \end{align}
        with $\ell\in(1,\infty]$, and for all $(t,j)\in\text{dom}(x)$. If this holds with $u\equiv0$, the set $\mathcal{A}$ is said to be \emph{Prescribed-Time Stable} via Flows (PT-S$_F$). \qedw
    \end{itemize}
    \end{definition}
\begin{remark}(Non-uniformity with respect to $\mu_0$)
    The constants $k_1, k_2>0$ in \eqref{eq:def:stability:HE} and the $\mathcal{KLL}$ function $\beta$ in the bound \eqref{eq:def:stability:PT} are independent of the initial conditions of $z$. However, the diffeomorphism $\mathcal{D}_{\mu_0,\ell}$ clearly depends on the initial value of $\mu$, which parameterizes the prescribed time $T_{\mu_0,\ell}$. Yet, the bounds \eqref{eq:def:stability:HE} and \eqref{eq:def:stability:PT} are uniform across the initial conditions of $x$, which is the main state of interest in the system.
\end{remark}
\section{Prescribed-Time and Hyperexponential Concurrent Learning with Switching datasets}
Consider a parameter estimation problem where the goal is to estimate $\theta^{\star}\in\mathbb{R}^n$ using real-time and past recorded measurements of the scalar signal
\begin{equation}\label{statement:parameterizedSignal}
\psi(\theta^{\star},t) = \phi(t)^\top\theta^{\star} + d(t),
\end{equation}
where $d:\mathbb{R}_{\ge0}\to\mathbb{R}$ is an unknown and Lebesgue measurable disturbance, and $\phi:\mathbb{R}_{\geq0}\to\mathbb{R}^n$ is a known \emph{regressor} that is assumed to be continuous and uniformly bounded.  

When the regressor $\phi$ satisfies a persistence of excitation (PE) condition, the typical approach to estimate $\theta$ relies on a ``standard'' recursive least-squares method \cite{Narendra:87_PE}. On the other hand, when $\phi$ does not satisfy a PE condition, but the practitioner has access to ``sufficiently rich'' past recorded data of $\phi$, concurrent learning (CL) can be used to estimate the true parameter $\theta^*$ recursively. In general, depending on their transient performance, there are three types of CL algorithms:
\begin{enumerate}
\item Classic exponentially stable CL algorithms, introduced in \cite{chowdhary2010concurrent,chowdhary2011theory}, and which achieve exponential rates of convergence proportional to the level of richness of the data. This approach extends the classic gradient-based parameter estimation algorithm using a ``data-driven'' or ``batch-based'' term that incorporates the recorded data. Other recent approaches that use initial excitation conditions (IEC) \cite{basu2019novel} and integral CL achieve similar properties \cite{le2022integral}.
\item High-order exponentially stable CL algorithms, studied in \cite{ochoa2021accelerated,le2022concurrent}, which incorporate momentum and sometimes resets (i.e., restarting) to improve transient performance. Such algorithms can achieve exponential rates of convergence proportional to the square root of the level of richness of the data, which is advantageous whenever the level of richness of the recorded data is ``low''.
\item Fixed-time stable CL algorithms (which generalized finite-time convergence methods), studied in \cite{rios2017time,ochoa2021accelerated} and \cite{tatari2021fixed}, which use non-smooth dynamics to achieve exact convergence to the true parameter before a fixed-time that is independent of the initial conditions of the algorithm, but inversely proportional to the level of richness of the data.
\end{enumerate}
As mentioned above, each of the existing approaches has different advantages and limitations. In this section, we introduce a novel CL architecture that complements the existing methods by offering new convergence properties that have not been studied before in the context of CL \cite{zimenko2023stability} and prescribed-time stability \cite{song2017time}. Such properties induce either faster-than-exponential convergence or convergence to the true parameter as time approaches a prescribed time, the latter being independent of the richness of data used by the algorithm. Moreover, we show that such properties can be achieved even when the algorithm implements, during bounded periods of time (which could be potentially unknown and persistent), data sets that are completely uninformative or corrupted. In practice, this situation emerges when the recorded data is persistently updated during the implementation of the CL algorithm, leading to different batches of available data for the algorithm, some of which might be compromised or tampered with. Since the practitioner cannot anticipate whether or not the current or new data is ``sufficiently rich'' or has been corrupted, establishing stability guarantees under sporadic implementations of uninformative or corrupted data can inject confidence in the implementation of the techniques.
\subsection{The Estimation Problem}
To estimate $\theta^{\star}$, we consider the algorithm shown in Figure \ref{fig:PTCLScheme}, where $\theta\in\mathbb{R}^n$ is the estimate of the parameter; $q\in\mathbb{Z}_{\geq1}$ is a logic switching state that is kept constant between switchings of the algorithm;  $\mu\in\mathbb{R}$ is a dynamic gain whose evolution is governed by any of the flow maps described in \eqref{eq:muFlowmaps}, and  $\tau\in\mathbb{R}_{\geq0}$ is used to model any explicit dependence on the time $t$. Between switching times, the states $(\theta,\mu,q,\tau)$ evolve according to the following continuous-time dynamics: 
%
\begin{align}\label{maindaynamicsCL}
    \dot{\theta}&=  \mu\cdot\Omega_{q}\left(\theta, \tau\right),~~~\dot{\mu}=F_{\mu,\ell}(\mu),~~~\dot{q}=0,~~~\dot{\tau}= 1,
\end{align}
where $\Upsilon>0$ is a tunable gain, and  $\Omega_q$ is a mode-dependent mapping to be characterized, and $k\in \mathbb{R}$.
%
%
%
%
\begin{figure}[t]
    \centering
    \includegraphics[width=0.65\linewidth]{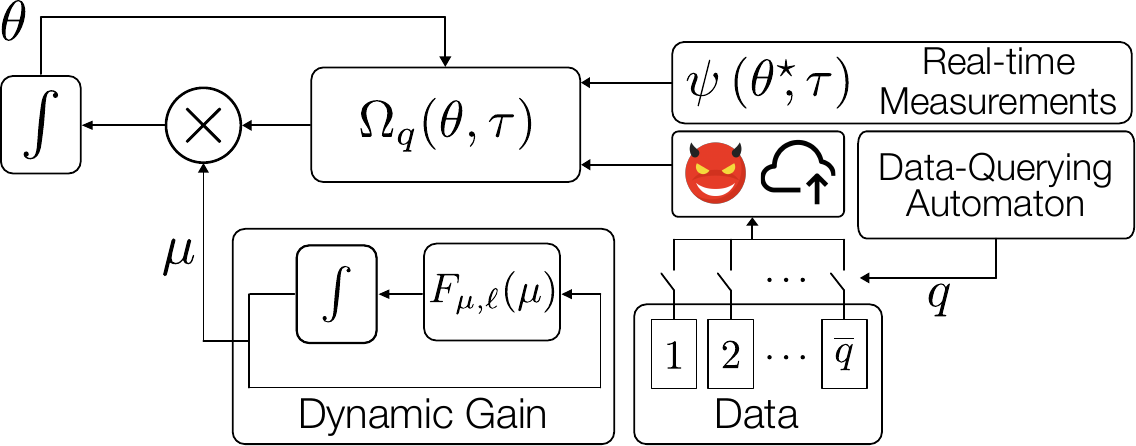}
    \caption{Block diagram of the proposed Switching Prescribed-Time and Hiperexponential Concurrent Learning algorithm. Recorded measurements (data) are transmitted remotely yet they can be susceptible to tampering or corruption.}
    \label{fig:PTCLScheme}
\end{figure}

To define the mapping $\Omega_q$ in \eqref{maindaynamicsCL}, as well as the switching dynamics of the algorithm, we let $\chi:\mathbb{R}^n\times\mathbb{R}_{\ge 0}\to \mathbb{R}^n$ be given by
\begin{equation}\label{def:Chi}
\chi(\theta, t)\coloneqq  \phi(t)\cdot\left(\phi(t)^\top\theta-\psi(\theta^{\star},t)\right),
\end{equation}
%
where $\phi$ and $\psi$ come from \eqref{statement:parameterizedSignal}. Note that $\chi(\theta, t)$ is a signal available to the practitioner at any time $t$. Indeed, in the context of concurrent learning, it is typically assumed that $(\phi,\chi)$ can be measured in real-time and that the algorithm has access to a sequence of \emph{recorded} measurements of $(\phi(\cdot),\chi(\theta,\cdot))$ taken at times $\{t_k\}$, $k\!\in\!\{1,2,\ldots, \bar{k}\}$. We are interested in CL applications where the recorded data can be segmented among multiple datasets, some of which might be ``uninformative'' or corrupted. 

\vspace{0.2cm}
\begin{remark}
In this paper, the use of the term ``recorded data'' should be considered in a more broad sense compared to traditional CL. In particular, since we will assume that the algorithm will switch between multiple data sets $\Delta_q$, a given data set implemented in the future after $T\!>\!0$ seconds could have been obtained after recording the data during the initial window of time $[0,T]$. Note that, on the other hand, this situation cannot occur in standard CL, where all recorded data must have been obtained before running the algorithm. However, to make our analysis tractable, we will impose \emph{a priori} assumptions on the data sets $\Delta_q$, including situations where such data sets are uninformative. To simplify notation, we let $\phi_{q,k}\!\coloneqq\! \phi(t_{q,k})$ and $\psi_{q,k}^\star\!\coloneqq\!\psi(\theta^\star, t_{q,k})$ denote the $kth$ regressor vector and signal measurement of the $qth$ dataset, respectively. \!\!\qedw 
\end{remark}
To formalize the model of the CL algorithm, we associate each dataset with a different value of the state $q\in \mathcal{Q}:=\{1,2,\ldots,\overline{q}\}$, with $\overline{q}\in\mathbb{Z}_{\ge 1}$. Therefore, we use $\mathcal{Q}$ to represent the set of indices of the multiple datasets, each dataset of the form $\Delta_q = (\Phi_q,\Psi_q)$, where $\Phi_q$ and $\Psi_q$ denote the $qth$ \emph{regressor data-matrix and signal data-vector}, respectively, and are defined as:
\begin{equation}\label{def:dataMatrix}
\Phi_q\coloneqq \sum_{k=1}^{\overline{k}_q}\phi_{q,k}\phi_{q,k}^\top, \quad \Psi_q\coloneqq \sum_{k=1}^{\overline{k}_q}\phi_{q,k}\psi_{q,k}^\star.
\end{equation}
The following definition, which is standard in the CL literature \cite{chowdhary2010concurrent}, formalizes the informativity properties of these datasets.

\begin{definition}(Sufficiently Rich Data)\label{definition:sufficientRich}
    A dataset $\Delta_q$ is said to be sufficiently rich (SR) if there exists $\alpha_q>0$ such that
    \begin{equation}\label{def:sufficientRich}
    \Phi_q \succeq \alpha_q I,
    \end{equation}
    where $\alpha_q$ denotes the level of richness of $\Delta_q$. If equation \eqref{def:sufficientRich} holds only with $\alpha_q=0$, then $\Delta_q$ is said to be insufficiently rich (IR). \qedw
\end{definition}
\begin{remark}\label{remark:corrupted}
    In IoT and edge computing applications, memory limitations of resource-constrained devices often necessitate remote storage and processing of datasets $\Delta_q$. During data transmission, these remote operations become vulnerable to man-in-the-middle attacks. Such attacks can corrupt the data $(\Phi_q,\Psi_q)$ of dataset $\Delta_q$, causing them to deviate from the structure defined in \eqref{def:dataMatrix}. For instance, even if $\Phi_q$ was originally sufficiently rich (SR), its positive definiteness property in \eqref{def:sufficientRich} may be lost—a scenario that could prove catastrophic for standard continuous-time CL algorithms.
    Thus, incorporating both the switching nature and potential corruption of these datasets into the convergence analysis of CL algorithms becomes essential for ensuring the practical applicability of CL techniques.\qedw
\end{remark}

In light of the considerations in Remark \ref{remark:corrupted}, we extend our analysis beyond the conventional SR or IR data classifications to include the utilization of corrupted data that may arise during operation. For simplicity, in this paper, we consider only corruption of the regressor data matrix $\Phi_q$, though the definition can be naturally extended to include corruption of the signal data vector $\Psi_q$.
\begin{definition}(Corrupted Data)
A dataset $\Delta_q$ is said to be corrupted if $\Phi_q$ is indefinite, negative definite, or $\Phi_q^\top \neq \Phi_q$. \qedw
\end{definition}
To differentiate the informative datasets from the uninformative or corrupt ones, we partition $\mathcal{Q}$ as follows:
\begin{equation*}
\mathcal{Q}=\mathcal{Q}_s\cup\mathcal{Q}_i\cup \mathcal{Q}_c,~~~\mathcal{Q}_s\cap(\mathcal{Q}_i\cup \mathcal{Q}_c)=\emptyset,~~~\mathcal{Q}_s,\mathcal{Q}_i,\mathcal{Q}_c\subset\mathbb{Z}_{\geq1},
\end{equation*}
where the modes in $\mathcal{Q}_s$ correspond to datasets that are SR, the modes in $\mathcal{Q}_i$ are IR, and the modes in $\mathcal{Q}_c$ employ corrupted data. To integrate the datasets into the CL algorithm, we let:
\begin{equation}\label{CLMap}
\begin{split}
\mathcal{R}_q(\theta)&\coloneqq \sum_{k=1}^{\overline{k}_q} \chi\left(\theta,t_{q,k}\right)=\Phi_q\theta -\Psi_q,~~\forall~q\in\mathcal{Q},\\
\Omega_q(\theta,\tau,u)&\coloneqq -k_{\text{t}}\chi\left(\theta,\tau,u\right) - k_{\text{r}}\mathcal{R}_q(\theta),
\end{split}
\end{equation}
where $k_{\text{t}},~k_{\text{r}}>0$ are tunable weights. System \eqref{maindaynamicsCL} with $\Omega_q$ as in \eqref{CLMap} describes the continuous-time dynamics of the proposed CL algorithm.
To guarantee that these dynamics can solve the parameter estimation problem we need to impose suitable conditions on the switching signals $q$. Namely, our goal is to characterize how frequently the CL algorithm can switch between multiple data sets (including those that are uninformative) such that stability and convergence are preserved. To answer this question we introduce a ``Data-Querying'' automaton that will characterize the family of switching signals that lead to stable closed-loop systems.
\subsection{The ``Data-Querying'' Automaton}
To model the switching dynamics of $q$, while simultaneously considering the goal of achieving hyperexponential (HE) or prescribed-time (PT) convergence to the true parameter $\theta^{\star}$, we consider a data-querying hybrid automaton that can be seen as an extension of the standard hybrid automaton that induces average dwell-time constraints on the switching signals \cite{cai2008smooth}. The data-querying automaton can be modeled as a hybrid dynamical system of the form \eqref{eq:HDS^Time^Gain} with no inputs, main state $y=(q,\rho_d,\rho_a)\in\mathbb{Z}_{\geq1}\times\mathbb{R}_{\geq0}\times\mathbb{R}_{\geq0}$, dynamic gain $\mu\in \mathbb{R}_{\ge1}$ and set-valued dynamics:
\begin{subequations}\label{DQAutomaton:HDS}
    \begin{align}
    (y,\mu)\in C_{\sigma}\times \mathbb{R}_{\ge1}&,~~~~~~\begin{pmatrix}\dot{y}\\\dot{\mu}
    \end{pmatrix}\in \mu\cdot F_{\sigma}(y)\times \{F_{\mu,\ell}(\mu)\},\label{DQAutomaton:flows}\\
    (y,\mu)\in D_{\sigma}\times \mathbb{R}_{\ge 1}&,~~~~\begin{pmatrix}y^+\\\mu^+
    \end{pmatrix} \in G_{\sigma}(y)\times \{\mu\},\label{DQAutomaton:jumps}
    \end{align}
where
 \begin{align}
    C_{\sigma}&\coloneqq \mathcal{Q}\times [0, N_0]\times [0, T_0],\qquad
     F_{\sigma}(y){\coloneqq} 
     \{0\}\times
     \left[0,\frac{1}{\tau_d}\right] \times
      \left(\left[0,\frac{1}{\tau_a}\right]{-}\mathbb{I}_{\mathcal{Q}_i\cup \mathcal{Q}_c}(q)
     \right),
     \label{DQAutomaton:flowmap}\\
    D_{\sigma}&\coloneqq \mathcal{Q}\times [1, N_0]\times [0, T_0],\qquad
    G_{\sigma}(y)\coloneqq \mathcal{Q}\setminus\{q\}\times
    \{\rho_d - 1\}\times
    \{\rho_a\},
    \label{DQAutomaton:jumpMap}
\end{align}
\end{subequations}
where $T_0>0$, $N_0\ge 1$, $\tau_d>0$, and $\tau_a>1$ are tunable parameters. 

\vspace{0.1cm}
 When $\mu\equiv 1$, hybrid automata of the form \eqref{DQAutomaton:HDS} are common in the study of asymptotic stability properties of switching systems under average dwell-time and average activation time constraints on the switching signals, see for example \cite{poveda2017framework,Liu_Tanwani_Liberzon2022}. However, the incorporation of the \emph{dynamic} gain $\mu$ in the continuous-time dynamics generates switching signals that differ from traditional average dwell-time signals. For example, as $t\to T_{\mu_0,\ell}$, system \eqref{DQAutomaton:HDS} allows for signals that switch at a faster rate. 
 
 The following Lemma provides some useful properties of the solutions to \eqref{DQAutomaton:HDS}. Its proof can be found in Section \ref{sec:Analysis}.
\begin{lemma}\label{lemma:ADT/AAT}
Let $y$ be a maximal solution to the HDS \eqref{DQAutomaton:HDS}, with $\mu(0,0)=\mu_0$. Then, the following holds:
\begin{enumerate}[(a)]
\item The total amount of flow time is bounded by $T_{\mu_0,\ell}$, i.e.,
\begin{align}\label{defprescribedtime}
\sup\{t\geq0:\exists~j\in\mathbb{Z}_{\geq0},(t,j)\in\text{dom}(y)\}
&\leq T_{\mu_0,\ell}.
\end{align}
\item For all $(t_1,j_1),(t_2,j_2)\in\text{dom}(y)$ such that $t_2\geq t_1$, the following Dilated Average Dwell-Time (D-ADT) condition is satisfied:
\begin{equation}\label{BUADT}
    j_2-j_1\leq \frac{1}{\tau_d}\left(\mathcal{D}_{\mu_0,\ell}(t_2) - \mathcal{D}_{\mu_0,\ell}(t_1)\right)+N_0,
\end{equation}
\item For all $(t_1,j_1),(t_2,j_2)\in\text{dom}(y)$ such that $t_2\geq t_1$, the following Dilated Average Activation Time (D-AAT)  condition is satisfied:
\begin{equation}\label{BUAAT}
\int_{t_1}^{t_2}\mu_\ell(t)\mathbb{I}_{\mathcal{Q}_i\cup \mathcal{Q}_c}(q(t,\underline{j}(t)))dt \leq \frac{1}{\tau_a}\left(\mathcal{D}_{\mu_0,\ell}(t_2)-\mathcal{D}_{\mu_0,\ell}(t_1)\right)+T_0,
\end{equation}
where $\underline{j}(t)=\min\{j\in\mathbb{Z}_{\geq1}:(t,j)\in\text{dom}(y)\}$.
\item For every hybrid arc $q:\text{dom}(q)\to\mathcal{Q}$ satisfying \eqref{BUADT} and \eqref{BUAAT} there exists a maximal solution $y$ to the HDS \eqref{DQAutomaton:HDS} having the same hybrid time domain. \qedw 
\end{enumerate}
\end{lemma}
When $\ell>1$, item (a) of Lemma \ref{lemma:ADT/AAT} states that every solution to the HDS \eqref{DQAutomaton:HDS} can flow for at most $T_{\mu_0,\ell}<\infty$ amount of time. Such time domains are typical in the context of PT- control \cite{Orlov2022,song2017time}, where the behaviors of the algorithms and controllers are mostly of interest during an initial finite window of time. In the context of PT-Stable algorithms, implementations are carried out by stopping the algorithm ``slightly'' before the prescribed time to avoid any singularity. Different techniques and heuristics for such implementations have been discussed in the literature, including the use of clipping, saturations, etc. We refer the reader to the recent works \cite{Orlov2022,Song2023,song2017time,todorovski2023practical,Hollowaylinearsystems2019} for a comprehensive discussion on practical applications and implementation. 
\begin{remark}
In \eqref{DQAutomaton:HDS}, the state $\rho_d$ acts as a timer that regulates how frequently $q$ switches between modes. Since in \eqref{DQAutomaton:flows} the flow map is multiplied by $\mu$, the rate of change of $\rho_d$ during flows is allowed to increase as $t\to T_{\mu_0,\ell}$. This leads to the D-ADT condition \eqref{BUADT}, where the right-hand side of the inequality grows to infinity as $t\to T_{\mu_0,\ell}$, thus allowing for more frequent switching as $t\to T_{\mu_0,\ell}$. \qedw 
\end{remark}
\begin{remark}
In the HDS \eqref{DQAutomaton:HDS} the state $\rho_a$ can be seen as a timer that regulates the amount of time spent by $q$ in the set $\mathcal{Q}_i\cup\mathcal{Q}_c$ compared to the time spent on $Q_s$. By construction, any maximal solution to this timer will induce the D-AAT bound \eqref{DQAutomaton:HDS}, thus limiting the proportion of time that the system spends in uninformative or corrupted modes. \qedw 
\end{remark}
\begin{remark}
We stress that the main goal of the Data-Querying automaton \eqref{DQAutomaton:HDS} is to provide a mathematical model to capture a family of switching signals under which the dynamics \eqref{maindaynamicsCL} achieve PT or HE convergence to the true parameter $\theta^\star$, i.e., it is used mainly for the purpose of analysis. Such signals might be \emph{unknown} to the practitioner who, even though might control the switching times, is aware a priori of the level of richness of the next data set $\Delta_{q^+}$. Again, this situation is common in practical applications where the data used by the CL algorithm is ``updated'' after a finite amount of time. Our hybrid model aims to capture such types of heuristics, which, to the best of our knowledge, had not been rigorously studied before the conference version of this paper \cite{ochoa2024hybridADHS}, let alone in the context of PT stability. \qedw 
\end{remark}
\subsection{Closed-Loop System and Main Result}
In this section, we study the stability properties of the complete dynamics that describe the Switching Prescribed-Time Concurrent and Hyperexponential Concurrent Learning (SPTHE-CL) algorithm shown in Figure \ref{fig:PTCLScheme}. 
The closed-loop system is a HDS of the form \eqref{eq:HDS^Time^Gain} with no inputs, state $z=(x,\tau,\mu)$, where $x=(\theta,y)\in\mathbb{R}^{n+3}$ represents the estimation and data querying states, with $\tau\in\mathbb{R}_{\ge0}$ being a variable used to model time, and $\mu\in\mathbb{R}_{\ge1}$ the dynamic gain. The system is allowed to evolve in continuous-time when the state $z$ belongs to the flow set
\begin{subequations}\label{HDS:tjTimeScale}
\begin{align}
    C_z\coloneqq \underbrace{\mathbb{R}^n\times C_{\sigma}}_{C_x}\times  \mathbb{R}_{\geq0} \times \mathbb{R}_{\ge1},
\end{align}
via the flow map $F:\mathbb{R}^{n+5}\times \mathbb{R}\rightrightarrows\mathbb{R}^{n+5}$:
\begin{align}\label{flows:tjTimeScale}
    \dot{z}=\begin{pmatrix}
    \dot{x}\\
    \dot{\tau}\\
    \dot{\mu}
    \end{pmatrix}\in F(z)\coloneqq\begin{pmatrix}
    \mu\cdot F_x(x,\tau)\\
    1\\
    F_{\mu,\ell}(\mu)
    \end{pmatrix},\quad \text{where}\quad      F_x(x,\tau)\coloneqq\begin{pmatrix}
        \Omega_q(\theta,\tau)\\
        F_{\sigma}(y)
        \end{pmatrix},
\end{align}
$\Omega_q$ is the learning map introduced in \eqref{CLMap}, and $C_\sigma$ and $F_\sigma$ are the flow set and flow map of the data-querying automaton defined in \eqref{DQAutomaton:flowmap}.
The system is allowed to evolve in discrete-time when the complete state $z$ belongs to the jump set
    \begin{align}
        D_z\coloneqq \underbrace{\mathbb{R}^n\times D_{\sigma}}_{D_x}\times  \mathbb{R}_{\geq0} \times \mathbb{R}_{\ge1},
    \end{align}
according to the jump map $G:\mathbb{R}^{n+5}\times\mathbb{R}\rightrightarrows\mathbb{R}^{n+5}$:
    \begin{align}
    z^+=\begin{pmatrix}
    x^+\\
    \tau^+\\
    \mu^+
    \end{pmatrix}\in G(z)\coloneqq\begin{pmatrix}
    G_x(x)\\
    \tau\\
    \mu
    \end{pmatrix},\quad\text{where}\quad  G_x(x)\coloneqq\begin{pmatrix}
        \theta\\
        G_{\sigma}(y)
        \end{pmatrix},
    \end{align}
\end{subequations}
and $D_\sigma$ and $G_{\sigma}$ are the jump set and the flow map of the data-querying automaton defined in \eqref{DQAutomaton:jumpMap}.
We can now state the main result of this paper, which provides HE and PT convergence results of the SPTHE-CL dynamics towards the true parameter $\theta^\star$. The proof leverages Lemma \ref{lemma:ADT/AAT} to establish HE-ISS$_F$ and PT-ISS$_F$ of a suitable surrogate hybrid dynamical system whose inputs capture the effects of disturbances and of corrupted data, and can be found in Section \ref{sec:Analysis}.
\begin{theorem}\label{thm:PTCL}
    Consider the SPTHE-CL algorithm in $\eqref{HDS:tjTimeScale}$ and suppose that:
    \begin{enumerate}[(a)]
    \item The set $\mathcal{Q}_s$ is not empty.
    \item The D-ADT and D-AAT parameters satisfy $\tau_d>0$ and $\tau_a> 1 + \frac{\varpi}{k_{\text{r}}\underline{\alpha}}$, where $\underline{\alpha}=\min_{q\in \mathcal{Q}_s}\alpha_q$, and $\varpi\coloneqq 1+ \max_{q\in \mathcal{Q}_c}\|\Phi_q\|$.
    \end{enumerate}
   Then, there exist $\kappa_1,\kappa_2>0$ such that for every solution $(z,u)=((\theta,y),\tau,\mu)$ to \eqref{HDS:tjTimeScale}, and every $\ell\in [1,\infty]$ the estimate $\theta$ satisfies:
    \begin{equation}\label{thm:bound}
    |\theta(t,j)-\theta^{\star}|\leq \kappa_1e^{-k_2\mathcal{D}_{\mu_0,\ell}(t)}e^{-\kappa_2j}\left|\vartheta_0\right|+\gamma_1\left(\sup_{0\le s\le t}|d(t)|\right) + \gamma_2\left(|d_r|\right) + \gamma_3\left(\sup_{q\in\mathcal{Q}_c}|\Phi_q \theta^\star - \Psi_q|\right),
    \end{equation}
    for all $(t,j)\in \text{dom}(z)$, where $\vartheta_0 = \theta(0,0) -\theta^{\star}$, $d_r$ is a vector that collects the noise in the recorded data for SR and IR datasets, and $\gamma_i$ is a positive definite function for $i\in\{1,2,3\}$.\qedw
\end{theorem}
\begin{remark}(On HE and PT-Convergence)
    When the disturbance function is uniformly bounded, i.e. $|d|_\infty<\bar{d}<\infty$, inequality \eqref{thm:bound} in Theorem \ref{thm:PTCL} establishes hyperexponential and prescribed-time ultimate boundedness results with respect to the true parameter $\theta^{\star}$ when $\ell=1$ and $\ell>1$, respectively, for the SPTHE-CL algorithm with switching datasets and partially corrupt data.  Additionally, when $\ell>1$, since $e^{-k_2\mathcal{D}_{\mu_0,\ell}(t)}\to 0^+$ as $t\to T_{\mu_0,\ell}<\infty$, the convergence to the ultimate bound is achieved in the prescribed-time $T_{\mu_0,\ell}$ assigned by the user via the choice of $\Upsilon>0$ and $\mu_0\ge 1$ (see Lemma \ref{lemma:DynamicGain}). 
    The ultimate bound depends on the bound $\bar{d}$ of the disturbance function as well as on the level of corruption of the data as specified by $\sup_{q\in\mathcal{Q}_c}|\Phi \theta^\star + \Psi_c|$. When no disturbance or corrupted data are present, the convergence to the true parameter is exact.
\qedw 
\end{remark}

\begin{remark}
Since by Lemma \ref{lemma:ADT/AAT} every solution to \eqref{HDS:tjTimeScale} satisfies the D-ADT and the D-AAT conditions, the result of Theorem \ref{thm:PTCL} can also be interpreted as a HE or PT convergence to the true parameter $\theta^\star$ under a switching system having a switching signal $q:\mathbb{R}_{\geq0}\to\mathcal{Q}$ that satisfies \eqref{BUADT} and \eqref{BUAAT}, where the left-hand side of \eqref{BUADT} corresponds to the total number of switches of $q$ in the interval $[t_1,t_2]$. \qedw 
\end{remark}

\begin{remark}
The result of Theorem \ref{thm:PTCL} can also be interpreted as a robustness result with respect to dataset corruption and degradation. Specifically, even when the algorithm uses datasets $\Delta_q$ that have been corrupted or degraded during transmission or storage, as long as these periods satisfy the  D-AAT \eqref{BUAAT}, PT or HE convergence to the true parameter will still be achieved. This robustness is particularly relevant for IoT applications where data corruption during transmission is a practical concern. ~\strut\hfill\qedw 
\end{remark}

\begin{remark}
The global uniform ultimate bound established in \eqref{thm:bound} also provides robustness guarantees with respect to bounded noise and disturbances acting on the measured signal \eqref{statement:parameterizedSignal}. Such results are fundamental for the feasibility of the implementation of parameter estimation algorithms. \qedw 
\end{remark}

\begin{remark}
As discussed in the literature \cite{Orlov2022,todorovski2023practical,song2017time,Song2023,Hollowaylinearsystems2019}, one of the main limitations of PT techniques that rely on dynamic gains is the need for ``stopping'' the algorithm before the prescribed time. While the bound \eqref{thm:bound} provides information on the residual bounds one can expect in such scenarios, an alternative approach is to implement hyperexponential convergence using $\ell=1$ in the gain dynamics. This choice leads to guaranteed completeness of solutions, avoids Zeno behavior, and circumvents the implementation challenges of PT algorithms by eliminating the need for algorithm termination or gain saturation, maintaining faster convergence than classical exponential results, albeit with the trade-off that the convergence time depends on initial conditions --- see Section \ref{sec:numexample} for a numerical illustration of these facts.
\qedw 
\end{remark}
\begin{remark}(On the use of PT-Algorithms in Estimation Problems)
Prescribed-Time stable algorithms were initially developed for regulation problems where the magnitude of the regulation error can be used as a ``signal'' to stop the PT dynamics. The lack of access to such error signals in parameter estimation problems is not problematic for the successful experimental implementation of PT-estimation dynamics, as shown in \cite{pan2022adaptive} for energy systems (e.g., batteries), the numerical study of Schrodinger systems \cite{steeves2020prescribed}, PT-estimation problems in quadrotor UAVS \cite{gong2023prescribed}, etc. In most scenarios, the proximity to the prescribed-time $T_{\mu_0,\ell}$ is used as a signal to clip the algorithm action. \qedw 
\end{remark}
\begin{remark}
    To the best knowledge of the authors, Theorem \ref{thm:PTCL} is the first HE  and PT-result in the context of CL, with the exception of the conference version of this paper that presented the result for $\ell=\infty$ \cite{ochoa2024hybridADHS}. Moreover, it is the first such result to handle switching datasets that contain both informative and potentially corrupted data.  In particular, Theorem \ref{thm:PTCL} can be seen as the HE and PT (and switching) counterpart of the recent \emph{fixed-time} stability results for CL algorithms established in \cite{ochoa2021accelerated,rios2017time} and \cite{tatari2021fixed}.
\end{remark}
We finish this section by discussing the case when only one (SR) dataset is available for the estimation of $\theta^{\star}$, the following corollary can be directly obtained. This result follows by simply taking the states $(q,\rho_d,\rho_a)$ as fixed constants in the HDS \eqref{HDS:tjTimeScale}.
\begin{corollary}
Consider the parameter estimation problem characterized by \eqref{statement:parameterizedSignal}, suppose that item (a) of Theorem \ref{thm:PTCL} holds and that $\mathcal{Q}=\{1\}$.
Then, there exist $\kappa_1,\kappa_2>0$ such that for every solution $x=(\theta,\tau,y)$ to \eqref{HDS:tjTimeScale}, the estimate $\theta$ satisfies \eqref{thm:bound} for all $(t,j)\in \text{dom}(x)$. \qedw 
\end{corollary}
\section{Analysis}\label{sec:Analysis}
To prove Theorem \ref{thm:PTCL}, we establish PT-ISS$_F$ and HE-ISS$_F$ of the set
\begin{equation}\label{def:AStable}
\mathcal{A}\coloneqq \mathcal{A}_x\times \mathbb{R}_{\ge 0}\times \mathbb{R}_{\ge 1},~\text{where}~ \mathcal{A}_x\coloneqq \{0\}\times C_\sigma,
\end{equation}
for a surrogate hybrid dynamical system $\mathcal{H}$ of the SPTHE-CL dynamics. This surrogate system is obtained by analyzing the estimation error and introducing inputs that capture disturbances and corrupted data. The state of the system is denoted by $\tilde{z}=(\tilde{x},\tau,\mu)$, where $\tilde{x}=(\vartheta,y)$ with $\vartheta\coloneqq\theta-\theta^{\star}$ denoting the parameter estimation error. The continuous-time dynamics of $\vartheta$ are given by $\dot{\vartheta}=\mu\cdot\mathcal{E}_q(\vartheta, \tau, u)$, where:
\begin{subequations}\label{ptCL:mapError}
\begin{equation}\label{eq:errorMap}
\mathcal{E}_q(\vartheta,\tau,u)=-\left(k_{\text{t}}\cdot\Xi(\tau) + \mathbb{I}_{\mathcal{Q}_s\cup \mathcal{Q}_i}(q)k_{\text{r}}\cdot\Phi_q\right)\vartheta  + \eta_q(\tau,u),\quad \forall (\vartheta,\tau)\in \mathbb{R}^n\times \mathbb{R}_{\ge0}
\end{equation}
and $u=(u_1,u_2,u_3)\in \mathbb{R}^{n_u}$ with: $u_1\in\mathbb{R}$ capturing the real-time disturbance $d(\tau)$; $u_2=(u_{2,q})_{q\in \mathcal{Q}_s\cup\mathcal{Q}_i}\in \mathbb{R}^{n_{u_2}}$ with $ n_{u_2}\coloneqq \sum_{q\in\mathcal{Q}_s\cup \mathcal{Q}_i}\bar{k}_q$ capturing disturbances during data recording of dataset $q$; and $u_3\in \mathbb{R}^{n^2}$ capturing corrupt data effects. The terms in \eqref{eq:errorMap} are defined as:
\begin{align}
\Xi(\tau)&\coloneqq \phi(\tau)\phi(\tau)^\top,\text{ and }
\eta_q(\tau,u)\coloneqq \begin{cases}
k_t\phi(\tau)\cdot u_1+k_r\sum_{k=1}^{\overline{k}_q}\phi_{q,k}\cdot (u_{2,q})_k \quad &\text{if }q\in \mathcal{Q}_s\cup \mathcal{Q}_i\\
k_t\phi(\tau)\cdot u_1 + k_r\cdot u_3&\text{if }q\in \mathcal{Q}_c.
\end{cases}
\end{align}
\end{subequations}
The vector field describing the evolution of $\vartheta$ for the original SPTHE-CL system is obtained from $\mathcal{E}_q(\vartheta,\tau,u)$ by setting $u_1=d(\tau)$, $(u_{2,q})_k = d(t_{q,k})$, and $u_{3} = -\Phi_q\theta^\star +  \Psi_q$.
Using \eqref{ptCL:mapError}, the surrogate HDS $\mathcal{H}$ has the form \eqref{eq:HDS^Time^Gain} with state $\tilde{z}$ and data $\mathcal{H}=(C_z\times \mathbb{R}^{n_u},\tilde{F},D_z\times \mathbb{R}^{n_u},G)$, where $\tilde{F}:\mathbb{R}^{n+5}\times\mathbb{R}^{n_u}\rightrightarrows\mathbb{R}^{n+5}$ is given by
\begin{equation*}
    \tilde{F}(\tilde{z},u)\coloneqq \begin{pmatrix}
    \mu\cdot\tilde{F}_x(\tilde{x},\tau,u)\\
    1\\
    F_{\mu,\ell}(\mu)
    \end{pmatrix},\quad \text{with}\quad
    \tilde{F}_x(\tilde{x},u)\coloneqq\begin{pmatrix}
    \mathcal{E}_q(\vartheta,\tau,u)\\
    F_{\sigma}(y)
    \end{pmatrix}.
\end{equation*}
Before presenting the proof of Lemma \ref{lemma:ADT/AAT} and Theorem \ref{thm:PTCL}, we introduce two auxiliary lemmas that extend \cite[Lemma 1]{ochoa2024dynamic}, which was stated without proof in that reference. These lemmas establish a correspondence between solutions of the HDS $\mathcal{H}$ and those of an associated target HDS $\hat{\mathcal{H}}$ via the diffeomorphisms $\mathcal{D}_{c,\ell}:[0,T_{c,\ell})\to \mathbb{R}_{\ge 0}$. This target system, has state $\hat{z}=(\hat{x},\hat{\tau},\hat{\mu})\in \mathbb{R}^{n+3}\times \mathbb{R}_{\ge 0}\times \mathbb{R}_{\ge1 }$, and is defined by:
\begin{subequations}\label{hds:target}
\begin{align}
    &(\hat{z},\hat{u})\in C_z\times \mathbb{R}^{n_u}, \quad \frac{\text{d}{
    \hat{z}
    }}{\text{d}s}=\begin{pmatrix}\frac{\text{d}{
        \hat{x}
    }}{\text{d}s}\\[3pt]
    \frac{\text{d}\hat{\tau}}{\text{d}s}\\[3pt]
    \frac{\text{d}{
            \hat{\mu}
    }}{\text{d}s}\end{pmatrix}\in \hat{F}(\hat{z},\hat{u})\coloneqq \tilde{F}_x(\hat{x}, \hat{u})\times \left\{\frac{1}{\hat{\mu}}\right\}\times \left\{\frac{1}{\hat{\mu}}F_{\mu,\ell}(\hat{\mu})\right\},\\[3pt]
    &(\hat{z},\hat{u})\in D_z\times \mathbb{R}^{n_u},\quad \hat{z}^+\in G(\hat{z}) \times \{\hat{\tau}\}\times  \{\hat{\mu}\},
\end{align}
\end{subequations}
where the derivative in $\frac{\text{d}\hat{z}}{\text{d}s}$ is taken with respect to the time variable $s$ to emphasize that the target system $\hat{\mathcal{H}}$ evolves in a different continuous-time scale parameterized by $s$. 
\begin{lemma}\label{lemma:EquivalenceOfSolutions:auxiliary}
    For every $z_0=(x_0,\tau_0,\mu_0)\in (C_x\cup D_x)\times \mathbb{R}_{\ge 0}\times \mathbb{R}$, it follows that
    \begin{enumerate}[(a)]
    \item  $(\hat{z},\hat{u})\in\mathcal{S}_{\hat{\mathcal{H}}}\left(z_0\right)\implies (\hat{z}\circ \mathbb{D}_{\mu_0,\ell},\hat{u}\circ \mathbb{D}_{\mu_0,\ell})\in\mathcal{S}_{\mathcal{H}}\left(z_0\right)$ with $\text{dom}(\hat{z}\circ \mathbb{D}_{\mu_0,\ell}) = \mathbb{D}_{\mu_0,\ell}^{-1}\left(\text{dom}(\hat{z})\right) $, and
    \item $(\tilde{z},u)\in\mathcal{S}_{\mathcal{H}}\left(z_0\right)\implies (\tilde{z}\circ \mathbb{D}^{-1}_{\mu_0,\ell},u\circ \mathbb{D}^{-1}_{\mu_0,\ell})\in\mathcal{S}_{\hat{\mathcal{H}}}\left(z_0\right)$ with $\text{dom}(\tilde{z}\circ \mathbb{D}^{-1}_{\mu_0,\ell}) = \mathbb{D}_{\mu_0,\ell}\left(\text{dom}(\tilde{z})\right) $,
    \end{enumerate}
    where $\mathbb{D}_{\mu_0,\ell}^{-1} \coloneqq \mathcal{D}_{\mu_0,\ell}^{-1}\times \text{id}_{\mathbb{Z}_{\ge 0}}$and where $\mathcal{D}_{c,\ell}$ is as defined in \eqref{eq:def:diffeomorphisms} for all $c\in \mathbb{R}_{\ge1}$, and all $\ell\in [1,\infty]$.\qedw
\end{lemma}
\begin{proof} ~\\
    (a) Let $z_0\coloneqq (x_0,\tau_0,\mu_0)\in (C_x\cup D_x)\times \mathbb{R}_{\ge0 }\times \mathbb{R}_{\ge 1}$, and $(\hat{z},\hat{u})$ be a maximal solution to the HDS $\mathcal{\hat{H}}$ from $z_0$. Then, by the definition of solutions to hybrid dynamical systems, for each $j\in\mathbb{Z}_{\geq0}$ such that the interior of $\hat{I}_j:=\{s\geq0:(s,j)\in \text{dom}(\hat{z})\}$ is nonempty, $\hat{z}$ satisfies
    \begin{equation}\label{lemma:maximalSolutionTarget}
    \frac{\text{d}}{\text{d}s}\hat{z}(s,j)\in  \tilde{F}_{x}\left(\hat{z}(s,j),\hat{u}(s,j)\right)\times \left\{\frac{1}{\hat{\mu}(s,j)}\right\} \times\left\{ \frac{1}{\hat{\mu}(s,j)}F_{\mu,\ell}(\hat{\mu}(s,j))\right\},
    \end{equation}
    for almost all $s\in \hat{I}_j$. Now, consider the hybrid signal $z\coloneqq\hat{z}\circ\mathbb{D}_{\mu_0,\ell}$. Then, using the chain rule:
    \begin{align}
    \frac{\text{d}}{\text{d}t}\tilde{z}\left(t,j\right) &= \frac{\text{d}}{\text{d}t}\left(\hat{z}\circ \mathbb{D}_{\mu_0,\ell}\right)(t,j)
        =\frac{\text{d}}{\text{d}s}\hat{z}(s,j)\Big|_{s=\mathcal{D}_{\mu_0,\ell}(t)} \frac{\text{d}\mathcal{D}_{\mu_0,\ell}}{\text{d}t}(t),\label{lemma:transformation:chainruleA}
    \end{align}
    for almost all $t\in\mathcal{D}^{-1}_{\mu_0,\ell}\left(\hat{I}_j\right)\eqqcolon I_j$.  Since $\frac{\text{d}\mathcal{D}_{\mu_0,\ell}}{\text{d}t}(t) = (\hat{\mu}\circ\mathcal{D}_{\mu_0,\ell})(t)$ for all $t\in[0,T_{\mu_0,\ell})$ via \eqref{eq:matchingEquation}, by noting that $\hat{\mu}$ does not change during the jumps of \eqref{hds:target}, and using the fact that $\text{sup}_{t}\text{dom}(\hat{z})=\infty$, since $\tilde{F}_{x}$ is Lipchitz in $\mathbb{R}^{n}\times (C_\sigma\cup D_{\sigma})$ and by desgin of the sets $C_z$ and $D_z$, it follows that $\mu(t,j)=\left(\hat{\mu}\circ \mathbb{D}_{\mu_0,\ell}\right)(t,j) = (\hat{\mu}\circ\mathcal{D}_{\mu_0,\ell})(t)$ for all $t\in I_j$. Therefore, using \eqref{lemma:maximalSolutionTarget} and \eqref{lemma:transformation:chainruleA} we obtain that
    \begin{align}
        \frac{\text{d}}{\text{d}t} \tilde{z}(t,j) &\in \left(\hat{\mu}\circ \mathbb{D}_{\mu_0,\ell}\right)(t,j)\tilde{F}_x\Big(\left(\hat{z}\circ \mathbb{D}_{\mu_0,\ell}\right)(t,j), \left(\hat{u}\circ \mathbb{D}_{\mu_0,\ell}\right)(t,j)\Big)\times \left\{1\right\}\times F_{\mu,\ell} \left(\left(\hat{\mu}\circ \mathbb{D}_{\mu_0,\ell}\right)(t,j)\right)\notag\\
        &=  \mu\left(t,j\right)\tilde{F}_x\left(z(t,j),u(t,j)\right)\times\{1 \}\times F_{\mu,\ell} \left(\mu(t,j)\right)\label{lemma:flows:A},
    \end{align}
    for almost all $t\in I_j$. Equation \eqref{lemma:flows:A} implies that $z$ satisfies the continuous-time dynamics of the HDS $\mathcal{H}$ for almost all $t\in I_j$. Moreover, note that by the definition of $I_j$, we directly get that $\mathcal{D}_{\mu_0,\ell}(\underline{t}_j)=\underline{s}_j$ and $\mathcal{D}_{\mu_0,\ell}(\overline{t}_j)=\overline{s}_j$ where $\underline{t}_j:=\min I_j$, $\overline{t}_j=\sup I_j$, $\underline{s}_j:=\min \hat{I}_j$, $\overline{s}_j=\sup \hat{I}_j$. Similarly,  we have that $\hat{x}(s,j+1)\in G(\hat{z}(s,j),\hat{u}(s,j))$ for every $(s,j)\in\text{dom}(\hat{z})$ such that $(s,j+1)\in\text{dom}(\hat{z})$, and therefore $x(t,j+1)\in G(z(t,j),u(t,j))$ since the HDS $\hat{\mathcal{H}}$ and the HDS $\mathcal{H}$ have the same discrete-time dynamics. Thus $z$ is a maximal solution to $\mathcal{H}$. Using these arguments,  we  also obtain that
    \begin{align*}
        \text{dom}\left(\hat{z}\circ \mathbb{D}_{\mu_0,\ell}\right)
                    &= \bigcup_{j=0}^{\sup_j\text{dom}(z)}I_j\times \{j\}
                    %
                    %
                    %
                    =\bigcup_{j=0}^{\sup_j\text{dom}(\hat{z})}\mathbb{D}_{\mu_0,\ell}^{-1}\left(\hat{I}_j\times \{j\}\right)
                    %
                    %
                    =\mathbb{D}_{\mu_0,\ell}^{-1}\left(\text{dom}(\hat{z})\right).
    \end{align*}        

    \vspace{0.1cm}\noindent 
    (b) Follows by the same ideas of the proof of item (a), via the use of the inverse mapping $\mathbb{D}_{c,\ell}^{-1}$.
    \strut\hfill$\blacksquare$
\end{proof}
Now, we use the result of Lemma \ref{lemma:EquivalenceOfSolutions:auxiliary} to obtain the following correspondence between maximal solution sets.
\begin{lemma}\label{lemma:EquivalenceOfSolutions:statement}
    Let $\mathbb{D}_{c,\ell}\coloneqq \mathcal{D}_{c,\ell}\times \text{id}_{\mathbb{Z}_{\ge 0}}$ for all $c\in\mathbb{R}_{\ge 1}$. Then, the map
    \begin{equation}\label{definition:W}
        \begin{split}
         \mathcal{W}:\quad~~\mathcal{S}_{\hat{\mathcal{H}}}\quad~~&\longrightarrow\quad \mathcal{S}_{\mathcal{H}}\\
                    \left(\hat{z}=(\hat{x},\hat{\tau},\hat{\mu}),\hat{u}\right)&\longmapsto (\hat{z}\circ \mathbb{D}_{\mu(0,0),\ell}, \hat{u}\circ \mathbb{D}_{\mu(0,0),\ell}),
        \end{split}
    \end{equation}
    is a bijection. Moreover, $\text{dom}\left(\mathcal{W}(\hat{z})\right) = \mathbb{D}_{\hat{\mu}(0,0),\ell}^{-1}\left(\text{dom}(\hat{z})\right)$ for every $(\hat{z},\hat{u})\in\mathcal{S}_{\hat{\mathcal{H}}}$.
    \strut\hfill\qedw
\end{lemma}
\begin{proof}
    Let $(\hat{z},\hat{u})\in \mathcal{S}_{\hat{\mathcal{H}}}$ be arbitrary. By the definition of solutions to hybrid dynamical systems with inputs, there exists $z_0=(x_0,\tau_0,\mu_0)\in (C_x\cup D_x )\times\mathbb{R}_{\ge0}\times \mathbb{R}_{\ge 1}$ such that $\hat{z}(0,0)=z_0$. Thus, using Lemma \ref{lemma:EquivalenceOfSolutions:auxiliary}, $\mathcal{W}(\hat{z},\hat{u}) = (\hat{z}\circ \mathbb{D}_{\mu_0,\ell}, \hat{u}\circ\mathbb{D}_{\mu_0,\ell})  \in \mathcal{S}_{\mathcal{H}}(z_0) \subset \mathcal{S}_{\mathcal{H}}$, meaning that $\mathcal{W}$ maps $\mathcal{S}_{\hat{\mathcal{H}}}$ into $\mathcal{S}_{\mathcal{H}}$, and is well defined.

To prove that $\mathcal{W}$ is a bijection, we consider the map
\begin{equation}\label{definition:What}
    \begin{split}
         \hat{\mathcal{W}}:\quad~~\mathcal{S}_{\mathcal{H}}\quad~~&\longrightarrow\quad \mathcal{S}_{\hat{\mathcal{H}}}\\
         \left(\hat{z}=(\hat{x},\hat{\mu}),\hat{u}\right)&\longmapsto (\hat{z}\circ \mathbb{D}_{\mu(0,0),\ell}, \hat{u}\circ \mathbb{D}_{\mu(0,0),\ell}),
    \end{split}
\end{equation}
which is well-defined by the results of Lemma \ref{lemma:EquivalenceOfSolutions:auxiliary}.
Then, for every $(\tilde{z},u)\in\mathcal{S}_{\mathcal{H}}$, we have that
\begin{align*}
\left(\mathcal{W}\circ\hat{\mathcal{W}}\right)(\tilde{z},u)=\mathcal{W}\left(\hat{\mathcal{W}}(\tilde{z},u)\right)
                                                %
                                                %
                                                &= \left(\tilde{z}\circ \mathbb{D}_{\mu(0,0),\ell}^{-1} \circ \mathbb{D}_{\mu(0,0),\ell},~u\circ \mathbb{D}_{\mu(0,0),\ell}^{-1} \circ \mathbb{D}_{\mu(0,0),\ell}\right)
                                                %
                                                %
                                                = (\tilde{z},u),
\end{align*}
which implies that $\mathcal{W}\circ \hat{\mathcal{W}} = \text{id}_{\mathcal{S}_{\mathcal{H}}}$.
Similarly, it follows that $\hat{\mathcal{W}}\circ\mathcal{W} = \text{id}_{\mathcal{S}_{\hat{\mathcal{H}}}}$. Hence, we have that $\mathcal{W}^{-1} =\hat{\mathcal{W}}$, and thus that $\mathcal{W}$ is a bijection between $\mathcal{S}_{\mathcal{H}}$ and $\mathcal{S}_{\hat{\mathcal{H}}}$. The fact that $\text{dom}\left(\mathcal{W}(\hat{z})\right) = \mathbb{D}_{\hat{\mu}(0,0),\ell}(\text{dom}(\hat{z}))$ follows directly from the results of Lemma \ref{lemma:EquivalenceOfSolutions:auxiliary}.\hfill\strut$\blacksquare$
\end{proof}
\begin{remark}
    Since the dynamics of $\hat{\mu}$ in $\hat{\mathcal{H}}$ are uncoupled from those of $\hat{x}$, $\tilde{F}_x$ is Lipschitz, $\frac{\text{d}}{\text{d}s}\hat{\tau}\le 1$, and the solutions of $\dot{\hat{\mu}}=\frac{1}{\hat{\mu}}F_{\hat{\mu},\ell}(\hat{\mu})$ are complete (see discussion above \eqref{eq:matchingEquation}), every solution $(\hat{z},\hat{u})\in \mathcal{S}_{\hat{\mathcal{H}}}$ is complete. 
        The completeness of solutions, together with the results of Lemma \ref{lemma:EquivalenceOfSolutions:statement} imply that when $s\to \infty$ for solutions in the target system $\hat{\mathcal{H}}$, $t\to T_{\mu_0,\ell}$ for the corresponding solutions of the original system $\mathcal{H}$. \strut\hfill$\blacksquare$
\end{remark}
%
%
We now present a proof for Lemma \ref{lemma:ADT/AAT} by using the results of Lemma \ref{lemma:EquivalenceOfSolutions:statement}.
\begin{proofLemmaADTAAT}~\\
    (a) Follows directy from Lemma \ref{lemma:DynamicGain}.\\
    (b) Let $\mathcal{H}_\sigma$ denote the HDS in \eqref{DQAutomaton:HDS} that defines the data-querying automaton. Similar to our approach with the full HDS $\mathcal{H}$, by dividing the flow-map of $\mathcal{H}_\sigma$ by $\mu$, we define a target HDS $\hat{\mathcal{H}}_{\sigma}$ that describes the evolution of the data-querying automaton in a dilated time-scale. This system has state $(\hat{y},\hat{\mu})$, where $\hat{y}=(\hat{q},\hat{\rho}_d,\hat{\rho}_a)\in \mathbb{R}^3$ and $\hat{\mu}\in\mathbb{R}_{\ge 1}$, with data given by
        \begin{subequations}\label{hybridBlowUpAutomatonUnstable}
        \begin{align}
        &(\hat{y},\hat{\mu})\in C_\sigma\times \mathbb{R}_{\ge 1},\qquad\qquad\begin{pmatrix}
            \frac{\text{d}\hat{y}}{\text{d}s}\\[4pt]
            \frac{\text{d}\hat{\mu}}{\text{d}s}
        \end{pmatrix}
        \in  \{0\}\times  \left[0,\dfrac{1}{\tau_d}\right]\times \left[0,\dfrac{1}{\tau_a}\right]-\mathbb{I}_{\mathcal{Q}_i\cup\mathcal{Q}_c}(\hat{q})\times \left\{\frac{1}{\hat{\mu}}F_{\mu,\ell}(\hat{\mu})\right\},\label{flowsHDS2}\\[5pt]
        &(\hat{y},\hat{\mu})\in D_\sigma\times \mathbb{R}_{\ge 1},\qquad\qquad\begin{pmatrix}
            \hat{y}^+\\\hat{\mu}^+
        \end{pmatrix}\in G_{\sigma}(\hat{y})\times \hat{\mu}.
        \end{align}
        \end{subequations}
        Since the dynamics of $\hat{y}$ in \eqref{hybridBlowUpAutomatonUnstable} are decoupled from those of $\hat{\mu}$, with jumps triggered only by the state $\hat{y}$, and noting that $\frac{\text{d}\hat{\mu}}{\text{d}s}=\frac{1}{\hat{\mu}}F_{\mu,\ell}(\hat{\mu})$ renders $\mathbb{R}_{\ge 1}$ forward-invariant, the hybrid time domains of solutions $(\hat{y},\hat{\mu})\in \mathcal{S}_{\hat{\mathcal{H}}_{\sigma}}$ are completely determined by the behavior of $\hat{y}$. Furthermore, as \eqref{hybridBlowUpAutomatonUnstable} incorporates both a dwell-time monitor $\hat{\rho}_d$ and a time-ratio monitor $\hat{\rho}_a$, it follows from  \cite[Ex. 2.15]{bookHDS} that every solution $(\hat{y},\hat{\mu})\in \mathcal{S}_{\hat{\mathcal{H}}_\sigma}$ has a hybrid time domain (HTD) that satisfies the Average Dwell-Time (ADT) bound:
        \begin{align}\label{proof:ADTBound}
        j_2-j_1\leq \frac{1}{\tau_d}(s_2-s_1)+N_0,
        \end{align}
        for all $(s_1,j_1),(s_2,j_2)\in\text{dom}(\hat{y})$ such that $s_2\geq s_1$.\smallbreak

        Now, note that the results of Lemma \ref{lemma:EquivalenceOfSolutions:statement} are also applicable to $\mathcal{H}_\sigma$ and $\hat{\mathcal{H}}_{\sigma}$, by dispensing with the variables $\theta$ and $\tau$, as their dynamics do not modify the evolution of the state $(y,\mu)$. \smallbreak
        
        Then, for every $(y_0,\mu_0)\in (C_\sigma\cup D_\sigma)\times \mathbb{R}_{\ge 1}$, and all $(y,\mu)\in \mathcal{S}_{\mathcal{H}_\sigma}((y_0,\mu_0))$ there exists $(\hat{y},\hat{\mu})\in \mathcal{S}_{\hat{\mathcal{H}}_\sigma}((y_0,\mu_0))$ such that $y =  \hat{y}\circ \mathbb{D}_{\mu_0,\ell}$, where $\mathbb{D}_{c,\ell}$ is as defined in Lemma \ref{lemma:EquivalenceOfSolutions:auxiliary}. Using this fact, and given that the diffeomorphism $\mathcal{D}_{\mu_0,\ell}$ is monotonically increasing, for all  $(t_1,j_1),(t_2,j_2)\in \text{dom}(y)$ with $0\leq t_1<t_2$, there exists unique $(s_1,j_1),(s_2,j_2)\in \text{dom}(\hat{y})$, such that $s_1=\mathcal{D}_{\mu_0,\ell}(t_1)$, $s_2=\mathcal{D}_{\mu_0,\ell}(t_2)$, and $0\leq s_1<s_2$. Thus, using the ADT bound \eqref{proof:ADTBound} for $(\hat{y},\hat{\mu})$ yields
        \begin{align}\label{proof:lemmaADT1}
        j_2-j_1&\leq \frac{1}{\tau_d}(\mathcal{D}_{\mu_0,\ell}(t_2)-\mathcal{D}_{\mu_0,\ell}(t_1))+N_0,
        \end{align}
        for all  $(t_1,j_1),(t_2,j_2)\in \text{dom}(y)$ with $0\leq t_1<t_2$.\medbreak

        \noindent (c) Via \cite[Lemma 7]{poveda2017framework}, it follows that every solution $(\hat{y},\hat{\mu})\in \mathcal{S}_{\hat{\mathcal{H}}_\sigma}$ has a hybrid time domain (HTD) that satisfies the Average Activation-Time (AAT) bound:
        \begin{equation}\label{proof:AATBound}
        \int_{s_1}^{s_2}\mathbb{I}_{\mathcal{Q}_i\cup \mathcal{Q}_c}(\hat{q}(s,\underline{\hat{\jmath}}(s)))ds \leq \frac{1}{\tau_a} (s_2-s_1)+T_0,
        \end{equation}
        for all $(s_1,j_1),(s_2,j_2)\in\text{dom}(\hat{y})$ such that $s_2\geq s_1$, where $\underline{\hat{\jmath}}(s)\coloneqq \min\left\{j\in\mathbb{Z}_{\ge 0}~:~(s,j)\in\text{dom}(\hat{q})\right\}$. Then, by using the same arguments as in item (b), for every $(y_0,\mu_0)\in (C_\sigma\cup D_\sigma)\times \mathbb{R}_{\ge 1}$, and all $(y,\mu)\in \mathcal{S}_{\mathcal{H}_\sigma}((y_0,\mu_0))$ there exists $(\hat{y},\hat{\mu})\in \mathcal{S}_{\hat{\mathcal{H}}_\sigma}((y_0,\mu_0))$ such that $y =  \hat{y}\circ \mathbb{D}_{\mu_0,\ell}$, and
        \begin{align}
            &\frac{1}{\tau_a} (\mathcal{D}_{\mu_0,\ell}(t_2)-\mathcal{D}_{\mu_0,\ell}(t_2))+T_0 \ge  \int_{s_1}^{s_2}\mathbb{I}_{\mathcal{Q}_i\cup \mathcal{Q}_c}(\hat{q}(s,\underline{\hat{\jmath}}(s)))ds\notag\\
            &\qquad\qquad= \int_{t_1}^{t_2}\frac{\text{d} \mathcal{D}_{\mu_0,\ell}(t)}{\text{d} t}\cdot \mathbb{I}_{\mathcal{Q}_i\cup\mathcal{Q}_c}\bigg(\hat{q}\Big(\mathcal{D}_{\mu_0,\ell}(t),\underline{\hat{\jmath}}\big(\mathcal{D}_{\mu_0,\ell}(t)\big)\Big)\bigg)dt
        =\int_{t_1}^{t_2}\mu_{\ell}(t)\cdot \mathbb{I}_{\mathcal{Q}_i\cup\mathcal{Q}_c}\big(q(t,\underline{j}(t))\big)dt,\label{integralAATtscale}
        \end{align}
        where we used the fact that
         \begin{equation*}
            q(t,\underline{j}(t))=  q\left(\mathcal{D}_{\mu_0,\ell}^{-1}(\mathcal{D}_{\mu_0,\ell}(t)),~\underline{j} \left(\mathcal{D}^{-1}_{\mu_0,\ell}\left(\mathcal{D}_{\mu_0,\ell}(t)\right)\right)\right)=\hat{q}(\mathcal{D}_{\mu_0,\ell}(t), \underline{\hat{\jmath}}\left(\mathcal{D}_{\mu_0,\ell}(t)\right)).
         \end{equation*}
        \noindent(d) Consider any hybrid arc $q$ satisfying \eqref{BUADT} and \eqref{BUAAT}. Using $\mathcal{D}_{c,\ell}^{-1}$, we can reverse the transformations used to obtain \eqref{proof:ADTBound} and \eqref{integralAATtscale}, which yields a hybrid arc $\hat{q}=\mathcal{W}^{-1}(q)$ satisfying \eqref{proof:lemmaADT1} and \eqref{proof:AATBound}. By \cite[Ex. 2.15]{bookHDS} and \cite[Lemma 7]{poveda2017framework}, every hybrid arc satisfying these bounds can be generated by the HDS $\hat{\mathcal{H}}_{\sigma}$. Then, Lemma \ref{lemma:EquivalenceOfSolutions:statement} ensures that the original arc $q$ can be generated by the HDS $\mathcal{H}_\sigma$.\strut\hfill$\blacksquare$        %
    \end{proofLemmaADTAAT}
    We are finally prepared to present the proof of the main result of this paper.
    \begin{proofTheoremPTCL} To obtain the result, we first analyze the stability of $\mathcal{A}$ under the target HDS $\hat{\mathcal{H}}$. To this end, let $W\left(\hat{z}\right)=\frac{1}{2}|\hat{\vartheta}|^2$. Since $\phi(\cdot)$ is uniformly bounded by assumption, there exists $\overline{\eta}>0$ such that $|\eta_{\hat{q}}(\hat{\tau},\hat{u})|\le \overline{\eta}|\hat{u}|$, for all $\hat{q}\in\mathcal{Q}$.  Then, for all  $(\hat{z},\hat{u})\in C_z\times \mathbb{R}$ and $\hat{f}\in \hat{F}(\hat{z},\hat{u})$, when the data in use is not corrupt, i.e., when $q\in \mathcal{Q}_s\cup \mathcal{Q}_i$, the function $W$ satisfies
        \begin{align}
            \left\langle \nabla W\left(\hat{z}\right),~\hat{f}\right\rangle &= -\hat{\vartheta}^{\top} \left(k_{\text{t}}\cdot\Xi(\hat{\tau}) + k_{\text{r}}\cdot\Phi_{\hat{q}}\right)\hat{\vartheta} - \hat{\vartheta}^{\top}\eta_{\hat{q}}(\hat{\tau},\hat{u})
            %
            %
            %
            \le -2\kappa_{\hat{q}}W\left(\hat{z}\right) + \left|\hat{\vartheta}\right|\overline{\eta}|\hat{u}|\label{proof:preDecresease}, 
        \end{align}
        where $\kappa_{\hat{q}}\!\coloneqq\! k_{\text{r}}\alpha_{\hat{q}}$,  and $\alpha_{\hat{q}}$ is the level of richness of the $\hat{q}^{th}$ dataset.  When $\hat{q}\in Q_s$, by using Young's inequality:
        \begin{align}
            \left\langle \nabla W\left(\hat{z}\right),~\hat{f}\right\rangle    
            %
            &\le  -\kappa_{\hat{q}}W(\hat{z}) + \frac{\overline{\eta}^2}{2\kappa_{\hat{q}}}|\hat{u}|^2,\label{proof:SRModes:wdot}
        \end{align}
        for all $\hat{f}\in \hat{F}(\hat{z}),~\hat{z}\in C$ and $\hat{q}\in\mathcal{Q}_s$. When the data is IR, inequality \eqref{proof:preDecresease} reduces to:
        \begin{align}\label{proof:IRModes:wdot}
            \left\langle \nabla W\left(\hat{z}\right), \hat{f}\right\rangle \le W(\hat{z}) + \frac{\overline{\eta}^2}{2}|\hat{u}|^2,
        \end{align}
        for all $\hat{f}\in \hat{F}(\hat{z}),~\hat{z}\in C_z,$ and $\hat{q}\in\mathcal{Q}_i$, and $\hat{u}\in \mathbb{R}^{n_u}$.
        Similarly, when the data is corrupt we obtain, for all $\hat{f}\in \hat{F}(\hat{z}),~\hat{z}\in C_z$, $u\in\mathbb{R}$, and $\hat{q}\in\mathcal{Q}_c$, that
        \begin{align}\label{proof:CorruptModes:wdot}
            \left\langle \nabla W\left(\hat{z}\right), \hat{f}\right\rangle \le \varpi W(\hat{z}) + \frac{\overline{\eta}^2}{2}|\hat{u}|^2,~~\text{ where}~~ \varpi\coloneqq 1+k_r\max_{q\in \mathcal{Q}_c}\|\Phi_q\|.
        \end{align}
        Let $\hat{\xi}\coloneqq (\underline{\kappa} + \varpi)\hat{\rho}_a$, with $\underline{\kappa}\coloneqq \min_{q\in\mathcal{Q}_s}\kappa_{\hat{q}}$. When $\hat{z}\in C_z$, the time derivative of $\hat{\xi}$ satisfies
        \begin{equation}\label{proof:dotxi}
            \frac{\text{d}\hat{\xi}}{\text{d}s} = (\underline{\kappa} + \varpi)\frac{\text{d}\hat{\rho}_{a}}{\text{d}s} \in [0, \zeta] - (\underline{\kappa} + \varpi)\mathbb{I}_{\mathcal{Q}_i\cup\mathcal{Q}_c}(\hat{q}),~~\text{where}~~\zeta\coloneqq \frac{1}{\tau_a}(\underline{\kappa} + \varpi).
        \end{equation}
         Now, consider the Lyapunov function $V(\hat{z})\coloneqq W(\hat{z})e^{\hat{\xi}}$. Using the fact that $|\hat{z}|_{\mathcal{A}} = |\hat{\vartheta}|$, it follows that $V$ satisfies the bounds $\underline{c}|\hat{z}|_{\mathcal{A}}^2 \le  V(\hat{z}) \le \overline{c}|\hat{z}|_{\mathcal{A}}^2$,
        %
        where $\underline{c}\coloneqq  1/2$, and $\overline{c}\coloneqq e^{(\varpi+\underline{\kappa})T_0}/2$.
        When $\hat{z}\in C$ and $\hat{q}\in\mathcal{Q}_s$, by using \eqref{proof:SRModes:wdot} and \eqref{proof:dotxi}, the time derivative of $V$ satisfies:
        \begin{align}
            \!\!\frac{\text{d}V}{\text{d}s} 
            &\le -\kappa_qW(\hat{z})e^{\hat{\xi}} + \frac{e^{(\varpi+\underline{\kappa})T_0}}{2\kappa_{\hat{q}}} \overline{\eta}^2|\hat{u}|^2+  V(\hat{z})\zeta
            \le -\left(\underline{\kappa} - \zeta \right)V(\hat{z}) +  \gamma|\hat{u}|^2,~~\text{where}~~\gamma \coloneqq \overline{c}\cdot\bar{\eta}^2/\min\{1,\underline{\kappa}\}.\label{proof:SRmodes:vdot}
        \end{align}
      Similarly, by using \eqref{proof:IRModes:wdot},\eqref{proof:CorruptModes:wdot}, and \eqref{proof:dotxi}, whenever $\hat{z}\in C_z$ and $\hat{q}\in\mathcal{Q}_i\cup\mathcal{Q}_c$, the time derivative of $V$ satisfies the bound \eqref{proof:SRmodes:vdot}.
        %
        %
        %
        %
        Since $\tau_a > 1 + \frac{\varpi}{k_{\text{r}}\underline{\alpha}}$, it follows that $\underline{\kappa} -\zeta >0$.
        On the other hand, when $\hat{z}\in D$, we have that $V_{\hat{q}^+}\left(\hat{z}^+\right) = W(\hat{z}^+)e^{\hat{\xi}^+}\leq V(\hat{z})$, i.e., $V$ does not increase during jumps. Using
        \cite[Lemma 9]{switchedPT}, we obtain:
         \begin{align}
        |\hat{z}(s,j)|_{\mathcal{A}}\leq \kappa_1|\hat{z}(0,0)|_{\mathcal{A}}e^{-\kappa_2(s+j)}+\kappa_3\sup_{0\le \nu\le s}|\hat{u}(\nu)|,\label{proof:ISS:stimescale}
        \end{align}
        with $\kappa_1 = e^{\frac{\lambda}{4}\frac{\tau_d}{1+\tau_d}\frac{(\varpi+\underline{\kappa})T_0}{2}N_0}$, $\kappa_2=\frac{\lambda}{4}\frac{\tau_d}{1+\tau_d}$, and $\kappa_3=2\left(\frac{\gamma}{\lambda}\right)^{\frac{1}{2}}$, where $\lambda\coloneqq\underline{\kappa}-\zeta>0$ by the assumption on $\tau_a$. 
        By Lemma \ref{lemma:EquivalenceOfSolutions:statement}, together with \eqref{proof:ISS:stimescale}, every maximal solution $(\tilde{z},u)\in \mathcal{S}_{\mathcal{H}}$ satisfies:
        \begin{equation}\label{thm:proof:prebound}
        |\tilde{z}(t,j)|_{\mathcal{A}}\leq \kappa_1|\tilde{z}(0,0)|_{\mathcal{A}}e^{-\kappa_2(\mathcal{D}_{\mu_0,\ell}(t)+j)}+\kappa_3\sup_{0\le \nu\le t}|u(\nu)|,
        \end{equation}
        which implies that $\mathcal{A}$ is HE-ISS$_F$ for $\mathcal{H}$ when $\ell=1$, and PT-ISS$_F$ for $\mathcal{H}$ when $\ell \in (1,\infty]$.

        Finally, we obtain \eqref{thm:bound} directly from \eqref{thm:proof:prebound}, by noting that $\left|\left(y(t,j), \tau(t,j)\right)\right|_{\mathcal{A}_0} = 0$ for every solution $\tilde{z}$ of $\mathcal{H}$ and $(t,j)\in \text{dom}(\tilde{z})$ with $t\le T_{\mu_0,\ell}$, using the fact that fact that $\left|\vartheta\right| = \left|\theta - \theta^{\star}\right|$, and setting $u_1=d(\tau)$, $(u_{2,q})_k = d(t_{q,k})$, and $u_{3} = -\Phi_q\theta^\star +  \Psi_q$.\hfill $\blacksquare$
    \end{proofTheoremPTCL}
%
%
\section{Numerical Example}
\label{sec:numexample}
\begin{figure*}[t]
    \centering
    \includegraphics[width=0.99\linewidth]{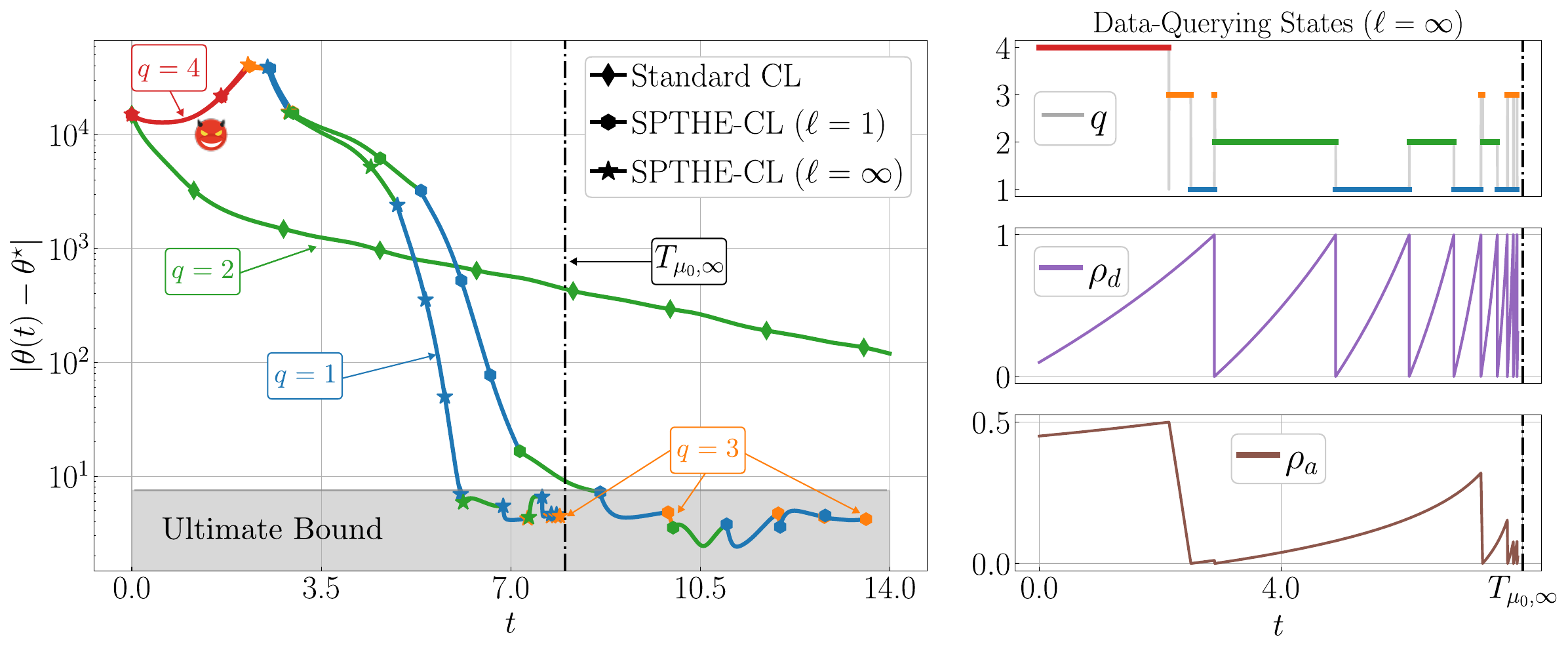}
    \caption{Comparison between standard concurrent learning with fixed dataset $\Delta_2$, and switching prescribed-time ($\ell=\infty$) and hyperexponential ($\ell=1$) concurrent Learning with informative data $\mathcal{Q}_s = \{1,2\}$, uninformative data $\mathcal{Q}_i=\{3\}$, and corrupt data $\mathcal{Q}_c=\{4\}$. The true parameter is $\theta^{\star} = (1,-2,1)$, and the prescribed-time $T_{\mu_0,\infty}=8$ for the $\ell=\infty$ case.}
        \label{fig:ptCL}
\end{figure*}
To illustrate the performance of the proposed SPTHE-CL algorithm and the bounds established in Theorem \ref{thm:PTCL}, consider the scalar-valued signal:
\begin{equation*}
\psi(\theta^{\star},\tau) = \left(\sin(\tau) - 1\right)^2 + d(\tau),
\end{equation*}
with $d(\tau)=\frac{1}{4}\tanh(\tau)$, and $|d|_\infty = \bar{d} =\frac{1}{4}$. Here, $\psi$ takes the form of \eqref{statement:parameterizedSignal} with $\theta^{\star} = (1,-2,1)$ and $\phi(\tau) = (1,\sin(\tau),\sin(\tau)^2)$. We consider four datasets $\{\Delta_q\}_{q=1}^4$ recorded at sequences $\{t_{1,k}\}_{k=1}^{\overline{k}_1} = \left\{0,-\frac{\pi}{2},-\frac{3\pi}{2}\right\}$, $\{t_{2,k}\}_{k=1}^{\overline{k}_2} = \left\{0,-\frac{\pi }{4},-\frac{7\pi}{4}\right\}$, $\{t_{3,k}\}_{k=1}^{\overline{k}_3}= \{0, -\pi , -2\pi\}$, and $\{t_{4,k}\}_{k=1}^{\overline{k}_4}= \{0, -\frac{\pi}{7} , -\frac{\pi}{5}\}$. Datasets $\Delta_1$ and $\Delta_2$ are SR with richness levels $\alpha_1 =0.44$ and $\alpha_2=0.15$, while $\Delta_3$ is IR, and $\Delta_4$ is corrupted with data matrix $\Phi_4 = [0.6,~ 0.3,~ 0.4; 0.3,~ 1,~ 0.3; 0.7,~0.5,~0.4].$

We implement the SPTHE-CL dynamics for both $\ell=1$ (HE case) and $\ell=\infty$ (PT case), with $\Upsilon=8$ and $\mu_0$ chosen for prescribed-time $T_{\mu_0,\infty}=8$, and with $k_t=k_r=1$. 
The switching signal $q$ satisfying the D-ADT and D-AAT conditions, is generated using the Data-Querying automaton \eqref{DQAutomaton:HDS} with parameters $T_0 = 1$, $N_0=2$, $\tau_d = 2$, and $\tau_a = 25$. These parameters satisfy the conditions of Theorem \ref{thm:PTCL}, namely $\tau_d > 0$ and $\tau_a > 1 + \frac{\varpi}{\underline{\alpha}}$, where $\underline{\alpha} = \min_{q\in\mathcal{Q}_s} \alpha_q = 0.15$ and $\varpi = 1 + \|\Phi_q\|$.

In Figure \ref{fig:ptCL}, we present the trajectories of the estimation error $|\theta-\theta^\star|$. For the case $\ell=\infty$, we plot the switching sequence $q$ and its associated states $\rho_d$ and $\rho_a$, where the increasing switching frequency as $t$ approaches $T_{\mu_0,\infty}$ is admissible under the D-ADT and D-AAT conditions \eqref{BUADT} and \eqref{BUAAT}. The gray area in the left subfigure depicts the uniform ultimate bound established in Theorem \ref{thm:PTCL} when $|d|_{\infty}<\bar{d}$.
For comparison, we also implement a standard CL algorithm with constant scaling $\mu\equiv1$ and fixed dataset $\Delta_2$ (system \eqref{HDS:tjTimeScale} without the Data-Querying automaton). As shown in the figure, for $\ell=\infty$ the estimation error rapidly approaches the ultimate bound as $t$ approaches $T_{\mu_0,\infty}$, even with IR or corrupted data. This demonstrates the algorithm's robustness to dynamic dataset changes during execution. The $\ell=1$ case shows similar fast convergence but continues beyond the prescribed-time, avoiding finite blow-up while maintaining rapid convergence. In contrast, the standard CL exhibits significantly slower convergence and fails to reach the ultimate bound within the simulation timeframe.

\section{Conclusions}
In this paper, we presented a novel switched prescribed-time and hyperexponential concurrent learning (CL)  algorithm that addresses two critical challenges in CL implementations: the dependence of convergence rates on dataset richness and the practical need to handle multiple, potentially corrupted datasets. The proposed framework unifies hyperexponential and prescribed-time convergence for CL dynamics under a single dynamic gain structure, allowing practitioners to either enhance exponential convergence rates or to prescribe exact convergence times that are independent of initial conditions and dataset informativity.

Our theoretical analysis leverages tools from hybrid dynamical systems to model switching signals through an autonomous data-querying automaton while incorporating a dilation/contraction transformation on hybrid time domains that proves crucial for handling both finite-time blow-up and exponentially growing gains. The convergence and stability guarantees of our algorithm are preserved even when operating with datasets that are not sufficiently rich or have been corrupted, making it particularly relevant for IoT and edge computing applications where data integrity cannot always be guaranteed.

Our results extend previous work on blow-up gains and bijections between solution sets to accommodate ISS-type bounds, providing a theoretical framework for analyzing prescribed-time and hyperexponential CL algorithms in the presence of noise and corrupted data. Future research directions include experimental validation under various practical scenarios, such as noisy sensors, limited computation resources, and saturated actuators.

\bibliographystyle{elsarticle-num}
\bibliography{references.bib}

\end{document}